\documentclass[11pt,reqno]{amsart}
 \usepackage{graphicx}
\usepackage{amscd,amsmath,amsopn,amssymb,amsthm,multicol}
\usepackage[color,matrix, all, 2cell]{xy}
\usepackage{amscd}
\usepackage{lscape}
\usepackage{slashed}
\usepackage{graphicx}
\usepackage{setspace}
\usepackage{upgreek}
\usepackage{textgreek}
\usepackage{enumerate}
\usepackage{color}
\usepackage{lscape}
\usepackage{tikz}
\usepackage{multirow}
\usepackage{cancel}
\usepackage{soul}
\usepackage{harmony}
\usepackage{comment}
\usepackage{wasysym}
\usepackage{mathrsfs}
\usepackage{mathtools}

\numberwithin{equation}{section}

\DeclareMathOperator{\Ad}{Ad}
\DeclareMathOperator{\ad}{ad}
\DeclareMathOperator{\Aut}{\mathsf{Aut}}
\DeclareMathOperator{\End}{End}

\DeclareMathOperator{\Iso}{\mathsf{Iso}}
\DeclareMathOperator{\RF}{\mathsf{RF}}

\DeclareMathOperator{\Id}{\mathsf{Id}}

\DeclareMathOperator{\Ric}{\mathsf{Ric}}
\DeclareMathOperator{\ric}{\mathsf{ric}}

\DeclareMathOperator{\Scal}{\mathsf{Scal}}
\DeclareMathOperator{\R}{\mathbb{R}}
\DeclareMathOperator{\M}{\mathscr{M}}

\DeclareMathOperator{\C}{\mathbb{C}}

\DeclareMathOperator{\ke}{\mathsf{Ker}}

\DeclareMathOperator{\RP}{\R\mathsf{P}}
\DeclareMathOperator{\CP}{\C\mathsf{P}}

\DeclareMathOperator{\Lie}{\mathsf{Lie}}

\DeclareMathOperator{\rnk}{\mathsf{rank}}

\DeclareMathOperator{\Spin}{\mathsf{Spin}}
\DeclareMathOperator{\SO}{\mathsf{SO}}
\DeclareMathOperator{\Sp}{\mathsf{Sp}}
 \DeclareMathOperator{\SU}{\mathsf{SU}}

\DeclareMathOperator{\U}{\mathsf{U}}

\DeclareMathOperator{\G}{\mathsf{G}}
\DeclareMathOperator{\F}{F}
\DeclareMathOperator{\Oo}{\mathsf{O}}
\DeclareMathOperator{\E}{\mathsf{E}}
\DeclareMathOperator{\Ss}{\mathsf{S}}

\DeclareMathAlphabet{\mathpzc}{OT1}{pzc}{m}{it}

\definecolor{dblue}{rgb}{0.01,0.01,0.44}
\definecolor{red}{rgb}{0.57,0.11,0.15}

\newcommand{\fr}{\mathfrak}
\newcommand{\al}{\alpha}
\newcommand{\be}{\beta}

\newcommand{\bb}{\mathbb}

\newcommand{\thickline}{\noalign{\hrule height 1pt}}

\theoremstyle{plain}
\newtheorem{lemma}{Lemma} [section]
\newtheorem{theorem}[lemma]{Theorem}
\newtheorem{corol}[lemma] {Corollary}
\newtheorem{prop} [lemma]{Proposition}

\theoremstyle{definition}
\newtheorem{definition}[lemma] {Definition}
\newtheorem{example}[lemma] {Example}

\newtheorem{remark}[lemma] {Remark}
\newtheorem*{remark*}{Remark}

\definecolor{dark}{rgb}{0.18,0.18,0.68}
\definecolor{mydark}{rgb}{0.78,0.08,0.08}
\definecolor{crew}{rgb}{0.2,0.5,0.2}
\definecolor{mmg}{rgb}{0.31,0.50,0.23}
\definecolor{dblue}{rgb}{0.01,0.01,0.44}
\definecolor{red}{rgb}{0.57,0.11,0.15}
\definecolor{cobalt}{RGB}{61,89,171}
\usepackage[colorlinks,citecolor=cobalt,linkcolor=cobalt,urlcolor=cobalt,pdfpagemode=UseNone,backref = page]{hyperref}

 \input ulem.sty
 
\title[Ancient solutions of the homogeneous  Ricci flow]{Ancient solutions of the homogeneous Ricci flow  on flag manifolds}

\author{Stavros Anastassiou} 
\address{Center for Research and Applications of Nonlinear Systems (CRANS), Department of Mathematics, University of Patras, Rion 26500, Greece}
\email{sanastassiou@gmail.com}

\author{Ioannis Chrysikos} 
\address{Faculty of Science, University of Hradec Kr\'alov\'e, Rokitanskeho 62, Hradec Kr\'alov\'e 50003, Czech Republic}
\email{ioannis.chrysikos@uhk.cz}

\begin{document}   

\begin{abstract}
 For any  flag manifold $M=G/K$ of a compact simple Lie group $G$ we describe non-collapsing ancient invariant solutions of the homogeneous  unnormalized Ricci flow. Such solutions  emerge from   an invariant Einstein metric  on $M$,  and   by \cite{BLM17} they must develop  a Type I singularity in their extinction  finite time, and  also to the past.    To illustrate the situation we engage ourselves with the global  study of the dynamical system induced by  the   unnormalized Ricci flow on any flag manifold $M=G/K$  with  second Betti number  $b_{2}(M)=1$, for a generic initial invariant metric. We  describe the corresponding dynamical systems and  present   non-collapsed ancient solutions, whose   $\alpha$-limit set  consists of fixed points at infinity of $\M^G$. Based on the Poincar\'{e}  compactification method, we show that these fixed points correspond to invariant Einstein metrics and  we   study their stability properties, illuminating thus the structure of the system's phase space.
%\\ 
%Keywords: Ricci flow, homogeneous spaces, flag manifolds, ancient solutions,  scalar curvature \\
%2010 MSC: Primary: 53C44, 53C25, 53C30, Secondary: 37C10
\end{abstract}

\maketitle

%%%%%%%%%%%%%%%%%%%%%%%%%%%%%%%%%%%%%%%%%%%%%%%%%%%%%%%%%%%%%%%%%%%%%%%%%%%%%%%%%%%%%%%%%%%%%%%%%%
%%%%%%%%%%%%%%%%%%%%%%%%%%%%%%%%%%%%%%%%%%%%%%%%%%%%%%%%%%%%%%%%%%%%%%%%%%%%%%%%%%%%%%%%%%%%%%%%%%
%%%%%%%															INTRODUCTION															%%%%%%
%%%%%%%%%%%%%%%%%%%%%%%%%%%%%%%%%%%%%%%%%%%%%%%%%%%%%%%%%%%%%%%%%%%%%%%%%%%%%%%%%%%%%%%%%%%%%%%%%%
%%%%%%%%%%%%%%%%%%%%%%%%%%%%%%%%%%%%%%%%%%%%%%%%%%%%%%%%%%%%%%%%%%%%%%%%%%%%%%%%%%%%%%%%%%%%%%%%%%

%===========================================================================================================
 %%%%%%%%%%%%%%%%%%%%%%%%%%%%%%%%%%%%%%%%%%%%%%%%%%%%%%%%%%%
 %%%%%%%%%%%%%%%%%%%%%%%%%%%%INTRODUCTION%%%%%%%%%%%%%%%%%%
 %%%%%%%%%%%%%%%%%%%%%%%%%%%%%%%%%%%%%%%%%%%%%%%%%%%%%%%%%%%%%
\section*{Introduction}

Given a Riemannian manifold $(M^{n}, g)$, recall that  the {\it  unnormalized  Ricci flow} is  the geometric flow defined by
\begin{equation}\label{flow1}
 \frac{\partial g(t)}{\partial t}=-2\Ric_{g(t)}\,, \quad g(0)=g\,,
\end{equation}
where  $\Ric_{g(t)}$ denotes the Ricci tensor of   the one-parameter family $g(t)$.
 The above system consists of  non-linear second
order partial differential equations on the open convex cone  $\M$ of Riemannian metrics on $M$.  A smooth family   $\{g(t) : t\in [0, T)\subset \bb{R}\}\in\M$  defined  for some $0<T\leq\infty$, is said to be a solution  of the Ricci flow   with initial metric $g$, if it satisfies the system (\ref{flow1}) for any $x\in M$ and $t\in[0, T)$.  The Ricci flow was introduced in the celebrated work of Hamilton \cite{Ham} and nowadays is the essential  tool in the proof of the famous {\it Poincar\'e conjecture} and {\it Thurston's geometrization conjecture}, due to  the seminal works \cite{Pe1, Pe2} of G. Perelman.

In general, and since a system of partial differential equations is involved,   it is hard to produce  explicit examples of Ricci flow solutions.
However,  the Ricci flow for an initial {\it invariant metric} reduces to  a system of ODEs.  More precisely,  homogeneity implies bounded curvature  (see \cite{chen}), and thus the  isometries of the initial metric will be in fact the isometries of any involved metric. Hence, when $g$ is an invariant metric,  any solution $g(t)$ of (\ref{flow1}) is also invariant. As a result,  in some cases it is possible to solve the system explicitly and proceed to  a study of their asymptotic properties, or even specify analytical properties related to different type of singularities  and deduce curvature estimates, see  \cite{BohmW07,  Grama09, AC11, Grama12,  Buz,   Laf15, Bohm15, BLM17, AbN16,  Grama20}, and the  articles quoted therein. Especially for the non-compact case, note that during the last decade the  Ricci flow for {\it homogeneous}, or {\it cohomogeneity-one}  metrics, together with  the so-called  {\it bracket flow} play a key   role  in  the study of the {\it  Alekseevsky  conjecture}, see   \cite{Payne1, Lauret11, Lauret11b, Lauret13,   LL14a, LL14, BL18a, BL18}.

In this work we examine the Ricci flow,   on compact homogeneous spaces with {\it simple spectrum of isotropy representation}, in terms of Graev \cite{Gr06,  Gr13}, or of {\it monotypic isotropy representation}, in terms  of Buzano \cite{Buz}, or Pulemotov  and Rubinstein \cite{Pule}.  Nowadays, such spaces are of special interest due to their rich applications  in the theory  of homogeneous Einstein metrics, prescribed Ricci curvature, Ricci iteration, Ricci flow and other (see  \cite{AP, Sak99, Bohm04,  BWZ04, Gr06,  Chry210, Chry211, Buz, Gr13, CS1, Bohm15,  Pod, Pule,  Grama20}).   Here, we  focus on  {\it  flag manifolds} $M=G/K$ of a compact simple Lie group $G$ and examine   the dynamical system induced by the vector field corresponding to the homogeneous Ricci flow equation.  On such cosets (even for $G$ semisimple), the homogeneous   Ricci flow  cannot possess fixed (stationary) points, since by  the theorem of Alekseevsky and Kimel'fel'd \cite{AK75} invariant Ricci flat metrics must be flat and so they cannot exist.  However,  as follows from the maximum principle, the homogeneous Ricci flow  on any flag manifold $M=G/K$ must admit  {\it ancient invariant solutions},  which in fact by the work of B\"{o}hm \cite{Bohm15} must have finite extinction time (note that since any flag manifold  is  compact and simply connected and carries invariant Einstein metrics, e.g. invariant K\"ahler-Einstein metrics always exist, the first conclusion above occurs also by a result of Lafuente \cite[Cor.~4.3]{Laf15}).
 
 To be more specific, recall that a solution $g(t)$ of the Ricci flow   is called {\it ancient} if it has as interval of definition the open set $(-\infty, T)$, for some  $T<\infty$.   Such solutions are important, because they arise as limits of blow ups of singular solutions to the Ricci flow near finite time singularities, see \cite{EMT11}.  For our case it follows  that for  any flag manifold  $M=G/K$  one must be always  able to specify a $G$-invariant metric $g$, such that any (maximal) Ricci flow solution $g(t)$ with  initial condition  $g(0)=g$  has an interval
of definition of the form $(t_{a},  T)$, with $-\infty\leq t_{a}<0$ and $T<\infty$.  Indeed, for any flag space below  we will provide  explicit solutions of  this type,  which are {ancient};  
 They  arise by using an invariant K\"ahler-Einstein metric, which always exists, or any other possible existent invariant Einstein metric $g_{0}$, and  they are defined on open intervals of the form $(-\infty,  T)$, where 
 \[
 T=\frac{1}{2\lambda}=\frac{n}{2}\Scal(g_{0})^{-1},
 \]
  with $\lambda=\frac{\Scal(g_{0})}{n}>0$ being the corresponding Einstein constant (see Proposition \ref{ak}). All such solutions   become extinct when $t\to T$, in the sense that $g(t)\to 0$, i.e. they  tend to  $0$. As ancient solutions,  they have  positive scalar curvature $\Scal(g(t))$ (\cite{Ben06}) and the asymptotic behaviour  of $\Scal(g(t))$, at least for the trivial one, can be easily treated, see  Proposition \ref{ak}  which forms a specification of \cite[Thm.~1.1]{Laf15}  on flag manifolds (see also  Examples \ref{s2c}, \ref{3compexa}).  Moreover, for any invariant Einstein metric on $M=G/K$ with $b_2(M)=1$, we show  the existence of an unstable manifold and compute its dimension. For the invariant Einstein metrics which are not K\"{a}hler, we obtain a 2- or 3-dimensional unstable manifold,  depending on the specific case,   which implies the existence of non--trivial ancient solutions emerging from the corresponding Einstein metric. Actually, by \cite{BLM17},  it  also follows that  these ancient solutions   develop a Type I singularity.   This means  (see \cite{EMT11, Buz, Bohm15,  PLu17})
 \[
 \lim_{t\to-\infty}\Big(|t|\cdot{\rm sup}_{x\in M}\|\mathsf{Rm}(g(t))\|_{g(t)}(x, t)\Big)<\infty\,,
 \]
 where $\mathsf{Rm}(g(t))$ denotes the curvature tensor of $(M=G/K, g(t))$, or equivalently that 
  there is a constant $0<C_{g_{0}}<\infty$ such that   $(T-t)\cdot{\rm sup}_{M}\|\mathsf{Rm}(g(t))\|_{g(t)}\leq C_{g_{0}}$, 
 for any $t\in(-\infty, T=\frac{1}{2\lambda})$.
Finally, by \cite{BLM17} we also deduce that the predicted non-trivial ancient solutions are {\it non-collapsed}  (and the same the trivial one, see Corollary \ref{noncollapsed}).  In this point we should mention that {\it not} any compact homogeneous space $M=G/K$ of a compact (semi)simple Lie group $G$ admits (unstable) Einstein metrics (see for example  \cite{WZ86, ParkSak97} for non-existence results). So, even assuming that the universal covering of  $M^n=G/K$ is not diffeomorphic to $\R^n$ (which for the compact case is equivalent to say  that $M^n$ is not a $n$-torus), the predicted solutions of Lafuente can be in general hard to be specified. %However, as we explained above, for flag manifolds, there are always solutions of  this type which can be described explicitly. 
 %For an illustration of the aforementioned  results, we    restrict our attention  to the class of flag manifolds $M=G/K$ with second Betti number equal to one, i.e. $b_2(M)=1$. 

% interval cannot be extended smoothly past time $T$ , the time $0<T\leq\infty$ is called {\it singular} (see \cite{Bohm15}). When  $T<\infty$ is {\it finite}, the Ricci flow
%solution is said to have {\it finite extinction time}, or {\it become extinct in finite time}.
  
%By  \cite{Laf15},  the  homogeneous Ricci flow has finite extinction
%time if and only if the scalar curvature of the evolved metrics becomes positive close
%to extinction time. Hence, as  a consequence, 

Note now that for any flag manifold $M=G/K$ of a compact simple Lie group $G$,  the symmetric space of invariant metrics $\M^{G}$ is the phase space of the homogeneous Ricci flow and it is flat, i.e. $\M^G\cong\R^{r}_{+}$ for some $r\geq 1$.
Therefore, the dynamical system of the homogeneous Ricci flow can be converted to a qualitative equivalent dynamical system of homogeneous polynomial equations and the well-known {\it  Poincar\'{e} compactification} (\cite{Poi}) strongly applies.  The main idea back of this method is to identify $\R^n$ with the northern and
southern hemispheres through central projections, and  then extend $X$ to a vector field $p(X)$ on $\Ss^n$ (see Section \ref{poinc}).   Here, for any  (non-symmetric)  {flag manifold} $M=G/K$ of a compact simple Lie group $G$ with $b_{2}(M)=1$, 
 we present the global study of   the dynamical system induced by the vector field corresponding to the unnormalized Ricci flow for an initial invariant metric, which is generic. In particular, the main contribution of this work is the description via the Poincar\'{e} compactification  method, of  the fixed points  of the homogeneous Ricci flow  at  the so-called {\it  infinity  of $\M^G$} (see Definition \ref{infin} and Remark \ref{fixinfinity}). Based on this method we can study  the stability properties of such fixed points, which we  prove that  are in bijective correspondence with  the existent  invariant Einstein metrics  on $M=G/K$, and moreover that coincide with  the $\alpha$-limit set of an invariant line, i.e. a solution of the homogeneous Ricci flow  which has as trace a line of $\M^G$.  It turns out that such solutions are ancient and non-collapsed, and develop Type I singularities.   Note that through the compactification procedure of  Poincar\'{e}, we are able to  distinguish   the unique invariant  K\"{a}hler-Einstein metric from the  other invariant Einstein metrics in terms of  (un)stable manifolds, in particular for every invariant Einstein metric we compute the dimension of the corresponding (un)stable manifold. Moreover, for the case $r=2$ we  discuss the  $\omega$-limit of  {\it any}  solution  of the homogeneous Ricci flow,  for details see Theorem \ref{mainthem}.  Since we are interested in the {\it unnormalized Ricci flow} on flag manifolds with  $b_{2}(M)=1$, we should finally mention that this dynamical system has been very recently examined by   \cite{Grama20}, for flag spaces with  three isotropy summands  (however via a  different method of the Poincar\'{e}  compactification),  while  a related study of   certain examples of flag manifolds with two isotropy  is given in \cite{Grama12}. Note finally that for flag manifolds with $r=2$, this work is complementary  to \cite{Buz}, in the sense that there were studied homogeneous ancient solutions  on compact homogeneous spaces with two isotropy summands; However, the  specific class of flag manifolds with $r=2$ was excluded (see at the end of the article \cite{Buz}), and the filling of this small gap was a motivation of the present work.

 The structure of the paper is given as follows: In Section \ref{prel} we refresh basics from the theory  of homogeneous spaces,   introduce the homogeneous Ricci flow and  recall some details from the structure and geometry of flag manifolds. Next, in Section \ref{ancient} we shortly present some basic results for ancient solutions emerging from  an invariant Einstein metric,  and also the Poincar\'{e}  compactification adapted to our scopes.  Finally, Section \ref{gloglo} is about the global study of the homogeneous Ricci flow for all non-symmetric flag manifolds $M=G/K$ of a compact simple Lie group $G$ with  $b_{2}(M)=1$, where a proof of  our main theorem, i.e. Theorem \ref{mainthem}, is presented.
 
\medskip
 \noindent {\bf Acknowledgements:}  
The authors are grateful to C. B\"{o}hm, R. Lafuente and Y. Sakane  for insightful comments.  They also thank J. Lauret and C. E. Will  for their  remarks  and pointing out a  mistake in the main  Theorem \ref{mainthem} in a previous draft of this manuscript.  I. C. acknowledges   support  by Czech Science Foundation, via the project GA\v{C}R no.~19-14466Y.  S. A.  thanks  the University of Hradec Kr\'alov\'e for hospitality along a research stay in September 2019.

 %===========================================================================================================
 %%%%%%%%%%%%%%%%%%%%%%%%%%%%%%%%%%%%%%%%%%%%%%%%%%%%%%%%%%%
 %%%%%%%%%%%%%%%%%%%%%%%%%%%%SECTION 1 PRELIMINARIES%%%%%%%%%%%%%%%%%%
 %%%%%%%%%%%%%%%%%%%%%%%%%%%%%%%%%%%%%%%%%%%%%%%%%%%%%%%%%%%%%
\section{Preliminaries}\label{prel}
We begin by  recalling  preliminaries of the homogeneous Ricci flow.  After that we will   refresh useful notions of the structure and geometry  of generalized flag manifolds.

%%%%%%%%%%%%%%%%%%%%%%%%%%%%SUBECTION 1.1 HRF%%%%%%%%%%%%%%%%%%
 %%%%%%%%%%%%%%%%%%%%%%%%%%%%%%%%%%%%%%%%%%%%%%%%%%%%%%%%%%%%%
 \subsection{Homogeneous Ricci flow}Recall that a {\it homogeneous Riemannian manifold} is a homogeneous space  $M=G/K$ (see  \cite{Kob2, Cheg, AVL} for details on homogeneous spaces) endowed with a $G$-invariant metric $g$, that is $\tau_{a}^{*}g=g$ for any $a\in G$, where $\tau : G\times G/K\to G/K$ denotes the transitive $G$-action. Equivalently, is a  Riemannian manifold $(M, g)$ endowed with  a transitive action of  its isometry group $\Iso(M, g)$.   If $M$ is connected, then each   closed subgroup $G\subseteq\Iso(M, g)$ which is transitive on $M$ induces a presentation of $(M, g)$ as a homogeneous space, i.e. $M=G/K$, where $K\subset G$ is the  stabilizer of some point $o\in M$.  In this case, the transitive Lie group $G$ can
  be also assumed to be connected (since the connected component of the identity
of $G$ is also transitive on  $M$). Usually,  to emphasize on the   transitive group $G$, we say that $(M, g)$ is a $G$-homogeneous Riemannian manifold. However, note that may exist many closed subgroups    of $\Iso(M,g)$ acting transitively on $(M, g)$. Next we shall  work with connected homogeneous manifolds.

As it is well-known,  the geometric properties of  a homogeneous space can be examined by restricting our attention to a point.  Set $o=eK$ for the identity coset  of $(M^n=G/K, g)$ and let $T_{o}G/K$ be the corresponding tangent space.    Since we assume the existence of a $G$-invariant metric $g$,  $K\subset G$ can be identified with a closed subgroup of $\Oo(n)\equiv\Oo(T_{o}G/K)$ (or of $\SO(n)$ if $G/K$ is oriented),  so $K$ is  compact  and hence  any homogeneous Riemannian manifold $(M^n=G/K, g)$ is  a {\it reductive homogeneous space}. This means that  there is a complement $\fr{m}$ of the Lie algebra $\fr{k}=\Lie(K)$ of the stabilizer $K$ inside the Lie algebra  $\fr{g}=\Lie(G)$  of $G$, which is  $\Ad_{G}(K)$-invariant, i.e. $\fr{g}=\fr{k}\oplus\fr{m}$ and $\Ad_{G}(K)\fr{m}\subset\fr{m}$,
where $\Ad_{G}\equiv\Ad : G\to\Aut(\fr{g})$ denotes the adjoint representation of $G$.  Note that in general the reductive complement $\fr{m}$  may not be   unique, and  for a general homogeneous space $G/K$ a sufficient condition for its existence is the compactness of $K$. On the other hand, once such a decomposition
has been fixed, there is always a natural identification of  $\fr{m}$ with the tangent space $T_{o}G/K=\fr{g}/\fr{k}$,    given by
\[
    X\in\fr{m} \longleftrightarrow X^{*}_{o} = \frac{d}{dt}\big|_{t=0}\tau_{\exp tX}(o)\in T_{o}G/K\,,
\]
where  $\exp tX$ is the one-parameter subgroup of $G$	 generated by $X$.  
 Under the linear isomorphism $\fr{m}=T_{o}G/K$,   the isotropy representation $\chi : K\to\Aut(\fr{m})$, defined by $\chi(k):=(d\tau_{k})_{o}$ for any $k\in K$,  is equivalent  with the representation $\Ad_{G}|_{K} : K\times\fr{m}\to\fr{m}$. Hence,    $\chi(k)X=\Ad_{G}(k)X$ 
  for any $k\in K$ and $X\in\fr{m}$.  In terms of Lie algebras we have $\chi_{*}(Y)X=[Y, X]_{\fr{m}}$ for any $Y\in\fr{k}$ and $X\in\fr{m}$, or in other words $\chi_{*}(Y)=\ad(Y)|_{\fr{m}}$.
 
The homogeneous spaces $M=G/K$ that we will   examine below (with $K$ compact),  are assumed to be almost effective, which means that  the kernel $\ke(\tau)$ (which is a normal subgroup both of $G$ and $K$),   is {\it finite}.   Thus, the isotropy representation $\chi$ is assumed to have a finite kernel, and then we may identify $\fr{k}$ with the Lie algebra $\chi_{*}(\fr{k})=\Lie(\chi(K))$ of the linear isotropy group  $\chi(K)\subset\Aut(\fr{m})$. When only the identity element $e\in G$  acts as the identity transformation on $M=G/K$, then the $G$-action is called {effective} and   the isotropy representation $\chi$ is injective. If we assume for example that $G\subseteq\Iso(M, g)$ is a closed subgroup, then  the action of $G$ to $G/K$  is effective.
Note that an almost effective action of $G$ on $G/K$ gives rise to an effective action of the group $G'=G/\ke(\tau)$ (of the same dimension with $G$), so we will not worry much for the effectiveness of $M=G/K$.  

Recall that the space of $G$-invariant symmetric covariant 2-tensors
on a  (almost) effective homogeneous space $M=G/K$ with a reductive decomposition $\fr{g}=\fr{k}\oplus\fr{m}$, is naturally isomorphic with the space of  symmetric bilinear forms on
$\fr{m}$, which are invariant under the isotropy action of $K$ on $\fr{m}$.  As a consequence, the space $\M^{G}$ of $G$-invariant Riemannian metrics on $M=G/K$ coincides with the space of inner products $\langle \ , \ \rangle$ on $\fr{m}$ satisfying
\[
\langle X, Y\rangle=\langle \chi(k)X, \chi(k)Y\rangle=\langle\Ad_{G}(k)X, \Ad_{G}(k)Y\rangle\,,
\]
for any $k\in K$ and $X, Y\in\fr{m}$. The correspondence is given by $\langle X, Y\rangle=g(X, Y)_{o}$. Moreover, when $K$ is compact and  $\fr{m}=\fr{h}^{\perp}$  with respect to $B=-B_{\fr{g}}$, where $B_{\fr{g}}$ is the Killing form of $\fr{g}$, then one can extend the above correspondence between elements $g=\langle \ , \ \rangle\in\M^G$  and $\Ad_{G}(K)$-invariant $B$-selfadjoint positive-definite endomorphisms $L : \fr{m}\to\fr{m}$ of $\fr{m}$, i.e. $\langle X, Y\rangle=B(LX, Y)$, for any $X, Y\in\fr{m}$. To simplify the text, whenever is possible next we shall relax the notation $\Ad_{G}(K)$ to $\Ad(K)$ and  scalar products on $\fr{m}$ as  above, will be just referred  to as $\Ad(K)$-invariant scalar products.  Note that the $\Ad(K)$-invariance of $\langle \ , \ \rangle$ implies its $\ad(\fr{k})$-invariance, which means that the endomorphism $\ad(Z)|_{\fr{m}} : \fr{m}\to \fr{m}$ is skew-symmetric with respect to $\langle \ , \ \rangle$, for any   $Z\in\fr{k}$.  When $K$ is connected,  the inclusions $\Ad(k)\fr{m}\subset\fr{m}$  and $[\fr{k}, \fr{m}]\subset\fr{m}$ are equivalent and hence one can pass from the $\Ad(K)$-invariance to $\ad(\fr{k})$-invariance and conversely. From now on will denote by $P(\fr{m})^{\Ad(K)}$ the   space  of all $\Ad(K)$-invariant inner products on  $\fr{m}$.

 Given a Riemannian manifold $(M^n, g)$, a solution  of the Ricci flow is a family of Riemannian metrics $\{g_{t}\}\in\M$
 satisfying the system (\ref{flow1}).
If the  initial metric $g=g(0)\in\M$ is a {\it $G$-invariant metric} with respect to some closed subgroup $G\subseteq\Iso(M, g)$,  i.e. $(M=G/K, g)$ is a  {\it homogeneous Riemannian manifold} and so $g\in\M^G$, then the solution $\{g(t)\}$ is called {\it homogeneous},   i.e.  $\{g(t)\}\in\M^G$. Indeed,  the   isometries of $g$ are  isometries for any other evolved metric, and by \cite{Kot10} it is known that the isometry
group is  preserved under the
Ricci flow.   Thus, after considering a reductive decomposition $\fr{g}=\fr{k}\oplus\fr{m}$ of $(M=G/K, g)$,  the homogeneity of $g$ allows us to reduce the Ricci flow to a system of  ODEs  for a curve of $\Ad(K)$-invariant inner products on $P(\fr{m})^{\Ad(K)}$, where $\fr{m}\cong T_{o}G/K$ is a reductive complement.  In particular,  due to the identification $\M^{G}\cong P^{\Ad(K)}(\fr{m})$ we may write $g(t)=\langle \ , \ \rangle_{t}$ and then    (\ref{flow1}) takes the form
\[
\frac{d}{dt}\langle \ , \ \rangle_{t}=-2\Ric_{\langle \ , \ \rangle_{t}}\,,\quad \langle \ , \ \rangle_{0}\equiv\langle \ , \ \rangle=g\,,
\]
where $\Ric_{\langle \ , \ \rangle_{t}}$ denotes the $\Ad(K)$-invariant bilinear form on $\fr{m}$, corresponding to the Ricci tensor of $g(t)$. Note that since $g_{0}=g(0)$ is an invariant metric, the  solution $g(t)$ of (\ref{flow1}) must be unique among
complete, bounded curvature metrics (see \cite{chen}).
\begin{remark}
When one is interested in  more general homogeneous spaces $M=G/K$ and a reductive decomposition may  not exist,  the above setting can be appropriately  transferred   to $\fr{g}/\fr{k}\cong T_{o}G/K$. However, the ``reductive setting'' serves well  the goals of this paper and it is sufficient for  our subsequent computations and description.
\end{remark}

%%%%%%%%%%%%%%%%%%%%%%%%%%%%SUBECTION 1.2 FLAG SPACES%%%%%%%%%%%%%%%%%%
 %%%%%%%%%%%%%%%%%%%%%%%%%%%%%%%%%%%%%%%%%%%%%%%%%%%%%%%%%%%%%
\subsection{Flag manifolds}
 Let $G$ be a compact semisimple Lie group with Lie algebra $\Lie(G)=\fr{g}$. A {\it flag manifold}\footnote{Also called {\it complex flag manifold}, or {\it generalized flag manifold}.}  is an adjoint orbit  of  $G$, i.e. 
 \[
 M=\Ad(G)w=\{\Ad(G)w: g\in G\}\subset\fr{g}
 \]    for some left-invariant vector field $w\in\fr{g}$.  Let $K=\{g\in G : \Ad(g)w=w\}\subset G$ be the isotropy subgroup of $w$ and let $\fr{k}=\Lie(K)$ be the corresponding Lie algebra. 
 Since $G$ acts on $M$ transitively, $M$ is diffeomorphic to the (compact) 
 homogeneous space $G/K$, that is $\Ad(G)w=G/K$. In particular, $\fr{k}=\{X\in\fr{g} : [X, w]=0\}=\ker\ad(w)$, where  $\ad :\fr{g}\to\End(\fr{g})$ is the adjoint representation of $\fr{g}$. Moreover, the set $S_{w}=\overline{\{\exp(tw) : t\in\mathbb{R}\}}$  
  is a torus in $G$ and the isotropy subgroup $K$ is 
 identified with the centralizer in $G$ of $S_{w}$, i.e. $K=C(S_{w})$.
 Hence $\rnk G=\rnk K$ and $K$ is connected. Thus, equivalently a  flag manifold  is a homogeneous space  of 
  the form $G/K$, where $K=C(S)=\{g\in G : ghg^{-1}=h  \ \mbox{for all} \  h\in S\}$ 
  is the centralizer of a torus $S$ in $G$.  When $K=C(T)=T$ is the centralizer of a maximal torus $T$ in $G$, the $G/T$ is called a {\it full flag manifold}. Flag manifolds admit a finite number of invariant complex structures, in particular flag spaces $G/K$ of a compact, simply connected, simple Lie group $G$ exhaust all compact, simply  connected, de Rham irreducible homogeneous K\"ahler manifolds (see for example \cite{AP,  Chry210,  Chry212, AC2019} for further details). 
  
So, for any flag manifold $M=G/K$ we may work with $G$ {\it simply connected} (if for instance a flag manifold of $\SO(n)$ is given,  one can always pass to its universal covering by  using the double covering $\Spin(n)$). For our scopes, it is also sufficient  to focus on the de Rham irreducible case, which is equivalent to say that $G$ is {\it simple}, see \cite{Kob2}. Hence, in the following we can always  assume  that $M=G/K$ satisfies these conditions, and   as before we shall denote by $B_{\fr{g}}$ the Killing form of  the Lie algebra $\fr{g}$.  The $\Ad(G)$-invariant inner product $B:=-B_{\fr{g}}$ induces a  bi-invariant metric on $G$ by left translations, and we may  fix, once and for all, a $B$-orthogonal $\Ad(K)$-invariant decomposition $\fr{g}=\fr{k}\oplus\fr{m}$.    $G$-invariant Riemannian metrics  on $G/K$
will be identified  with  $\Ad(K)$-invariant inner products
$\langle \ , \ \rangle$ on the reductive complement $\fr{m}=T_{o}G/K$.  
 Note that   the restriction $B\big|_{\fr{m}}$ induces the so-called Killing metric $g_{B}\in\M^G$, which is the unique invariant metric for which the natural projection $\pi : (G, B) \longrightarrow (G/K, g_{B})$  is a Riemannian submersion.

   The second Betti number of any   flag manifold $M=G/K$ is encoded in    the corresponding {\it painted Dynkin diagram}. To recall the procedure, let $ \fr{g}^{\mathbb{C}}=\fr{h}^{\mathbb{C}}\oplus\sum_{\al\in R}\fr{g}_{\al}^{\mathbb{C}}$
 be the usual root space decomposition of the complexification 
  $\fr{g}^{\mathbb{C}}$ of $\fr{g}$, with respect to a  Cartan subalgebra $\fr{h}^{\mathbb{C}}$ of $\fr{g}^{\mathbb{C}}$, 
 where $R\subset(\fr{h}^{\mathbb{C}})^{*}$ is the root system of $\fr{g}^{\mathbb{C}}$.    Via the Killing form of $\fr{g}^{\C}$ we identify $(\fr{h}^{\mathbb{C}})^{*}$ with $\fr{h}^{\mathbb{C}}$.   Let  $\Pi=\{\al_{1}, \ldots, \al_{\ell}\}$ \ $(\dim\fr{h}^{\mathbb{C}}=\ell)$ be a fundamental system of $R$ and choose a subset  $\Pi_{K}$   of $\Pi$. We denote by  $R_{K}=\{\be\in R : \be=\sum_{\al_{i}\in\Pi_{K}}k_{i}\al_{i}\}$  the closed subsystem spanned by $\Pi_{K}$. Then, the Lie subalgebra  $  \fr{k}^{\mathbb{C}}=\fr{h}^{\mathbb{C}}\oplus\sum_{\be\in R_{K}}\fr{g}_{\be}^{\mathbb{C}}$ is a reductive subalgebra of $\fr{g}^{\bb{C}}$, i.e.   it admits a decomposition of the form  $\fr{k}^{\bb{C}}=Z(\fr{k}^{\bb{C}})\oplus\fr{k}_{ss}^{\bb{C}}$,   where $Z(\fr{k}^{\bb{C}})$ is its center and   $\fr{k}_{ss}^{\bb{C}}=[\fr{k}^{\bb{C}}, \fr{k}^{\bb{C}}]$  the semisimple part of $\fr{k}^{\bb{C}}$. In particular,  $R_{K}$ is the root system of $\fr{k}^{\bb{C}}_{ss}$, and  thus $\Pi_{K}$ can be considered as the associated fundamental   system.  Let    $K$ be the connected Lie subgroup of $G$
  generated by $\fr{k}=\fr{k}^{\mathbb{C}}\cap\fr{g}$. Then the homogeneous manifold $M=G/K$ is a flag manifold, and any flag manifold is defined in this way, i.e.,  by the choise of a triple $(\fr{g}^{\bb{C}}, \Pi, \Pi_{K})$, see also \cite{AP, Chry210, Chry212, AC2019}.

      Set  $\Pi_{M}=\Pi\backslash \Pi_{K}$  and  $R_{M}=R\backslash R_{K}$, such that   $\Pi=\Pi_{K}\sqcup \Pi_{M}$, and $R=R_{K}\sqcup R_{M}$, respectively.  Roots in  $R_{M}$ are called  {\it complementary roots}. 
 Let $\Gamma=\Gamma(\Pi)$ be the Dynkin diagram of the fundamental 
 system $\Pi$. 
 \begin{definition}
Let $M=G/K$ be a flag manifold.  By painting  black the nodes  of $\Gamma$  corresponding    to  $\Pi_{M}$, we obtain the {\it painted Dynkin diagram}  of  $G/K$ (PDD in short).  In this diagram
    the subsystem $\Pi_{K}$ is determined as the subdiagram of white roots.
 \end{definition}
  \begin{remark}
Conversely, given a PDD, one may determine the associated flag manifold $M=G/K$   as follows: The group $G$ is defined as the unique simply connected Lie group generated by the  unique real   form $\fr{g}$   of the   complex simple Lie algebra $\fr{g}^{\bb{C}}$ (up to inner automorphisms of $\fr{g}^{\bb{C}}$),  which is reconstructed by the  underlying Dynkin diagram.  Moreover, the connected Lie subgroup $K\subset G$ is defined by using  the encoded by the PDD splitting $\Pi=\Pi_{K}\sqcup\Pi_{M}$; The semisimple part of $K$ is obtained from  the (not  necessarily connected) subdiagram of white simple roots,  while  each black root, i.e.   each  root in $\Pi_{M}$,  gives rise to a $\U(1)$-summand.    Thus,  the PDD determines 
 the  isotropy group $K$ and the space $M=G/K$ completely.        By using certain rules to determine  whether     different PDDs define isomorphic 
  flag manifolds (see \cite{AP}),  one  can obtain all flag manifolds $G/K$ of  a compact  simple Lie group $G$ (see for example the tables in \cite{AC2019}).
\end{remark}
    \begin{prop}\label{betti} \textnormal{(\cite{Bor, AP})}
    The second Betti number of  a    flag manifold $M=G/K$ equals  to the totality of black nodes in   the corresponding  PDD, i.e. the cardinality of the set $\Pi_{M}$.
    \end{prop}
    Note that for any flag manifold $M=G/K$, the $B$-orthogonal reductive complement $\fr{m}$ decomposes into a direct sum of $\Ad(K)$-inequivalent and irreducible submodules, which we call {\it isotropy summands}, see  \cite{AP, Chry210, AC2019} and the references therein. This means that  when  $\fr{m}\cong T_{o}M$ is viewed as a $K$-module,  then there is always   
 a  $B$-orthogonal $\Ad(K)$-invariant decomposition 
\begin{equation}\label{mmm}
 \fr{m}= \fr{m}_{1}\oplus\cdots\oplus\fr{m}_{r} \,,
 \end{equation}
for some $r\geq 1$, such that
\begin{itemize}
\item $K$ acts (via the isotropy representation)  irreducibly on any $\fr{m}_1, \ldots, \fr{m}_r$, and
\item  $\fr{m}_i\ncong\fr{m}_{j}$ are inequivalent as $\Ad(K)$-representations for any $i\neq j$. 
\end{itemize}
In fact,  a decomposition as in {\rm (\ref{mmm})} satisfying the given conditions must be unique,  up to a permutation of the isotropy summands, see for example \cite{Pule}.    When $r=1$,  $M=G/K$ is an isotropy irreducible compact Hermitian symmetric space (HSS in short), and all these cosets can be viewed as flag manifolds with $b_{2}(M)=1$. Note that from the class of full flag manifolds only $\CP^{1}\cong\Ss^{2}=\SU(2)/\U(1)$ is an irreducible HSS, and hence a flag manifold with $b_{2}(M)=1$. In this text we are mainly interested in  non-symmetric flag manifolds $M=G/K$ with $b_{2}(M)=1$, and in this case $r$ is bounded by  the inequalities $2\leq r\leq 6$ (see below). 
  \begin{lemma} \label{monotypic} \textnormal{(\cite{AP, Chry210})}
  Let $M=G/K$ be a flag manifold of a compact  simple Lie group. Then, the  isotropy representation  of $M$ is monotypic and decomposes as in {\rm (\ref{mmm})}, for some $r\geq 1$. Moreover,  any  $G$-invariant metric on $M=G/K$ is  given by 
\begin{equation}\label{diago}
 g=\langle \ , \ \rangle=\sum_{i=1}^{r}x_{i}\cdot B|_{\fr{m}_i}\,,
 \end{equation}
where $x_i\in\R_{+}$ are positive real numbers for any $i=1, \ldots, r$.  Thus,  $\M^G$ coincides with the open convex cone   $\R^{r}_{+}=\{(x_1,  \ldots, x_{r})\in\R^{r} : x_{i}>0 \ \text{for any} \  \ i=1 \ldots r\}$. 
\end{lemma}
Invariant metrics as in (\ref{diago}) are called {\it diagonal} (see \cite{WZ86} for details). By Schur's Lemma,  the Ricci tensor  $\Ric_{g}$ of such a diagonal invariant metric $g$, needs  to preserve  the splitting (\ref{mmm}) and consequently, $\Ric_g$ is also diagonal, i.e. $\Ric_{g}(\fr{m}_i, \fr{m}_j)=0$, whenever $i\neq j$.   As before, $\Ric_{g}$ is  determined by  a symmetric $\Ad(K)$-invariant bilinear form on    $\fr{m}$, although  not necessarily positive definite, and hence it has the expression
 \[
 \Ric_{g}\equiv\Ric_{\langle \ , \ \rangle} =\sum_{i=1}^{r}y_{i}\cdot B\big|_{\fr{m}_{i}} =\sum_{i=1}^{r}(x_{i} \cdot\ric_{i})\cdot B\big|_{\fr{m}_{i}}\,,
 \]
for some  $y_{i}=x_{i} \cdot\ric_{i}\in\mathbb{R}$, where $\langle \ , \ \rangle\in P^{\Ad(K)}(\fr{m})$ is the  $\Ad(K)$-invariant inner product corresponding to  $g\in\M^G$ and $\ric_{i}$ are the so-called {\it Ricci components}.
There is a simple description of   $\ric_{i}$, and hence of $\Ric_g$ (and also of the scalar curvature $\Scal_{g}={\rm tr}\Ric_{g})$, in terms of the  metric parameters $x_{i}$, the dimensions $d_{i}=\dim_{\R}\fr{m}_i$ and the so-called {\it structure constants   of $G/K$ with respect to the   decomposition {\rm (\ref{mmm})}}, 
\[
c_{ij}^{k}\equiv{k\brack {i \ j}}=\sum_{\al, \be, \gamma}B([X_{\al}, Y_{\be}], Z_{\gamma})^{2}\,,\quad i, j, k\in\{1, \ldots, r\},
\]
 where $\{X_{\al}\}$, $\{Y_{\be}\}$, $\{Z_{\gamma}\}$  are {\it $B$-orthonormal bases} of $\fr{m}_i$, $\fr{m}_j$, $\fr{m}_k$, respectively.   These non-negative  quantities were introduced in \cite{WZ86} and they have a long tradition  in the theory of  (compact) homogeneous Einstein spaces,  see for example \cite{ParkSak97, Sak99, Bohm04, BWZ04, CS1}. Following \cite{Gr13}, next we shall refer to rational polynomials depending on some real variables  $x_1, x_1^{-1}, \ldots, x_{m}, x_{m}^{-1}$ for some  positive integer $m$, by the term {\it Laurent polynomials}. If such a rational polynomial is homogeneous, then  it will be called a {\it homogeneous Laurent polynomial}.  Now, as a conclusion   of most general results presented by \cite{WZ86, ParkSak97}, one obtains the following
 \begin{prop}\label{Ricc}{\textnormal{(\cite{WZ86, ParkSak97})}}   Let $M=G/K$ be a flag manifold of a compact simple Lie group $G$. Then,\\
     1) The components $\ric_{k}$ of the Ricci tensor $\Ric_{g}$ corresponding to  $g=\sum_{i=1}^{r}x_{i}\cdot B|_{\fr{m}_{i}}\in\M^G$  are     homogeneous Laurent polynomials  in  $x_1, x_1^{-1}, \ldots, x_{r}, x_{r}^{-1}$  of degree -1, given by
\[
\ric_{k}=\frac{1}{2x_{k}}+\frac{1}{4d_{k}}\sum_{i, j=1}^{r}\frac{x_{k}}{x_{i}x_{j}}  {k \brack i\ j}-\frac{1}{2d_{k}}\sum_{i, j=1}^{r}\frac{x_{j}}{x_{k}x_{i}}   {j \brack k\ i} \,, \qquad (k=1, \ldots, r)\,.
\]
%   &=&\frac{1}{2x_{k}}+\frac{1}{4d_{k}}\sum_{i, j=1}^{r} {k \brack i\ j}\frac{{x_{k}}^{2}-2{x_{j}}^{2}}{x_{i}x_{j}x_{k}} \,, \qquad (k=1, \ldots, r)\,.
 %  \end{eqnarray*}
 2) The scalar curvature $\Scal_{g}=\sum_{i=1}^{r}d_{i}\cdot\ric_{i}$ corresponding to $g=\sum_{i=1}^{r}x_{i}\cdot B|_{\fr{m}_{i}}\in\M^G$ is a  homogeneous Laurent polynomial in $x_1, x_1^{-1}, \ldots, x_{r}, x_{r}^{-1}$  of degree -1,  given by
 \[
\Scal_{g} =\frac{1}{2}\sum_{i=1}^{r}\frac{d_{i}}{x_{i}}-\frac{1}{4}\sum_{i, j, m=1}^{r} {m  \brack i\ j}\frac{x_{m}}{x_{i}x_{j}}\,.
\]
Note that $\Scal_{g} : M\to\R$ is a constant function.
 \end{prop} 
 
 %%%%%%%%%%%%%%%%%%%%%%%%%%%%%%%%%%%%%%%%%%%%%%%%%%%%%%%%%%%
 %%%%%%%%%%%%%%%%%%%%%%%%%%%%SECTION 2 GLOBAL STUDY%%%%%%%%%%%%%%%%%%
 %%%%%%%%%%%%%%%%%%%%%%%%%%%%%%%%%%%%%%%%%%%%%%%%%%%%%%%%%%%%%
 %%%%%%%%%%%%%%%%%%%%%%%%%%%%%%%%%%%%%%%%%%%%%%%%%%%%%%%%%%%%%
 %%%%%%%%%%%%%%%%%%%%%%%%%%%%%%%%%%%%%%%%%%%%%%%%%%%%%%%%%%%%%
 
 \section{The  homogeneous Ricci flow on flag manifolds}\label{ancient}
In this section we fix a flag manifold  $M=G/K$  of a compact, simply connected, simple Lie group $G$, whose isotropy representation $\fr{m}$ decomposes as in (\ref{mmm}). 
 Since  $\M^G$  endowed with the $L^2$-metric coincides with the open convex cone $\R_{+}^r$,  for the study of the Ricci flow, as an initial invariant metric  we  fix the general invariant metric $g=\sum_{i=1}^{r}x_{i}\cdot B|_{\fr{m}_{i}}$; This  often  will be simply denoted  by $g=(x_1, \ldots, x_{r})\in\R^{r}_{+}\cong\M^{G}$.  Then, the Ricci flow equation (\ref{flow1}) with initial condition $g(0)=g$ descents to the following  system of ODEs (in component form):
\begin{equation}\label{homflow1}
\Big\{\dot{x}_k=-2x_{k}\cdot\ric_{k}(x_1, \ldots, x_r)\,, \quad  1\leq k\leq r\Big\}.
\end{equation}
Since  $g\in\M^{G}$,  and any  invariant metric  evolves  under the Ricci flow again to a $G$-invariant metric,   every solution of the homogeneous   Ricci flow needs to be of the form 
\[
g(t)=\sum_{i=1}^{r}x_{i}(t)\cdot B|_{\fr{m}_{i}}, \quad g(0)=g=\langle \ , \ \rangle\,,
\]
where the smooth  functions $x_{i}(t)$ are   positive on the same maximal interval  $t\in(t_a, t_b)$ for which  $g(t)$ is defined. Usually, solutions have a maximal interval  of definition $(t_a, t_b)$, where  $0\in(t_a, t_b)$, $-\infty\leq t_{a}<0$ and $t_{b}\geq+\infty$. Next we are mainly interested in {\it ancient solutions}.
\subsection{Invariant ancient solutions} Recall that (see for example \cite{Buz, Laf15})
\begin{definition}
 A solution $g(t)$ of (\ref{homflow1}) which is defined on an interval of the  form    $(-\infty, t_{b})$ with $t_{b}<+\infty$, is called {\it  ancient}.
\end{definition}
Note that ancient solutions typically arise as singularity models of the Ricci flow and it is well-known that all ancient solutions have non-negative scalar curvature, see for instance \cite[Cor.~2.5]{Ben06}.
We first proceed    with the following
\begin{lemma}\label{nofixed}
 The homogeneous Ricci flow  {\rm (\ref{homflow1})} on a flag manifold $M=G/K$, does not possess   fixed points  in $\M^G\cong\R_+^r$. In other words, considering 
 the associated flow to {\rm (\ref{homflow1})}, i.e. the map
 \[
 \psi_{t}: \M^G\to\M^G\,,\quad g\longmapsto g(t)
 \]
 where $t\in(-\epsilon, \epsilon)$ for some $\epsilon\in(0, \infty)$, there is no $g\in\M^{G}\cong\R^{r}_{+}$ such that  $\psi_{t}(g)=g$ for all $t$. 
 \end{lemma}
 \begin{proof}
Obviously, fixed points of the homogeneous Ricci flow  need to correspond to {\it invariant Ricci-flat metrics}, and conversely. Since $M=G/K$ is compact, according to  Alekseevsky and Kimel'fel'd  \cite{AK75} such a metric must be necessarily flat.
But then $M$ must be a torus, a contradiction.
\end{proof}
However, the system (\ref{homflow1})  can admit more general solutions, different than   stationary points.   Indeed,  any flag manifold  $M^n=G/K$   is compact and simply connected, and hence its universal covering is not diffeomorphic  with an  Euclidean space. Hence, by maximum principle, or by   \cite[Cor.~4.3]{Laf15}  it follows that
\begin{prop}\textnormal{(\cite{Laf15})} \label{sscal}
For any flag manifold $M=G/K$ of a compact,  simply connected, simple Lie group $G$, there exists a $G$-invariant metric $g$ such that any Ricci flow solution $g(t)$ with  initial condition  $g(0)=g$,  has an interval
of definition of the form $(t_{a}, t_{b})$, with $-\infty\leq t_{a}<0$ and $t_{b}<\infty$. 
\end{prop} 
In the following section, for all flag manifolds $M=G/K$ with $b_2(M)=1$  we  will construct solutions of   (\ref{homflow1}),   specify  the maximal interval of their definition $(t_{a}, t_{b})$ and study their asymptotic behaviour at the {\it infinity} of the corresponding phase space $\M^{G}$ (in terms of the Poincar\'{e} compactification, see Definition \ref{infin}). These solutions, both trivial and non-trivial, emerge from invariant Einstein metrics, in particular the trivial are   {\it shrinking type solutions}  of  (\ref{homflow1}), in terms of   \cite[p.~98]{chow} for instance.  Moreover, they  are {\it ancient} and hence the scalar curvature along such solutions  is positive, while they  extinct in finite time.  %In fact, by \cite{Laf15} a  homogeneous Ricci flow has finite extinction
%time, if and only if the scalar curvature of the evolved metrics becomes positive close
%to the extinction time. And indeed,   
In fact, for any flag manifold $M=G/K$ one can state the following basic result.
\begin{prop}\label{ak}
Let $M=G/K$ be a flag manifold  of a compact,  connected, simply connected, simple Lie group $G$, with $\M^{G}\cong\R^{r}_{+}$, for some $r\geq 1$. Let $g_{0}$ be any $G$-invariant Einstein metric on $M=G/K$ with Einstein constant $\lambda$ and consider the 1-parameter family $g(t)=(1-2\lambda t)g_{0}$. Then, \\
(1) $g(t)$ is an ancient solution of {\rm (\ref{homflow1})} defined on the open interval  $(-\infty, \frac{1}{2\lambda})$ with $g(t)\to 0$ as $t\to\frac{1}{2\lambda}$ (from below).  Hence, its scalar curvature $\Scal(g(t))$ is a monotonically increasing function with the same interval of definition, and satisfies
\[
\lim_{t\to-\infty}\Scal(g(t))=0\,,\quad \lim_{t\to\frac{1}{2\lambda}}\Scal(g(t))=+\infty.
\]
In particular, $\Scal(g(t))>0$ for any $t\in(-\infty, \frac{1}{2\lambda})$. \\
(2)  Similarly, the Ricci components  $\ric_{i}^{t}\equiv\ric_{i}(g(t))$ corresponding to  the solution $g(t)=(1-2\lambda t)g_{0}$, satisfy the following asymptotic properties:
\[
\lim_{t\to-\infty}\ric_{i}^{t}=0\,,\quad \lim_{t\to\frac{1}{2\lambda}}\ric_{i}^{t}=+\infty.
\]
In particular, $\ric_{i}^{t}>0$, for any $t\in(-\infty, \frac{1}{2\lambda})$.
\end{prop}
\begin{proof}
(1) Let $g_{0}$ be a $G$-invariant Einstein metric on $M=G/K$ (e.g. one can fix as $g_{0}$ an invariant K\"ahler-Einstein metric, which always exists).  Set $c(t)=(1-2\lambda t)$ and note that $c(t)>0$ for any  $t\in(-\infty, \frac{1}{2\lambda})$. Since $g(t)=c(t)g_{0}$ is a 1-parameter family of invariant metrics, we compute
\[
g(t)=\sum_{i=1}^{r}(1-2\lambda t)x_i^{0}\cdot B|_{\fr{m}_i}=\sum_{i=1}^{r}c(t)x_{i}^{0}\cdot B|_{\fr{m}_i}=\sum_{i=1}^{r}x_{i}(t)\cdot B|_{\fr{m}_i}\,,
\] 
where $x_{i}(t):=c(t)x_{i}^{0}$, for any $1\leq i\leq r$ with $x_{i}(0)=x_{i}^{0}$, where without loss of generality we assume that $g_{0}=(x_1^{0}, \ldots, x_r^{0})$ for some $x_{i}^{0}\in\R^{r}_{+}$ and some $\bb{N}\ni r\geq 1$.
  As it is well-known from the general theory of Ricci flow, and it is trivial to see, $g(t)$ is a solution of   {\rm (\ref{homflow1})} which has as interval of definition the open set $(-\infty, \frac{1}{2\lambda})$. Hence it is an ancient solution, since $0<\frac{1}{2\lambda}<+\infty$ (recall that $\lambda>0$ is the  Einstein constant of an invariant Einstein metric on a compact homogenous space), and moreover $\lim_{t\to-\infty}g(t)=+\infty$ and $\quad \lim_{t\to\frac{1}{2\lambda}}g(t)=0.$
   On the other hand, by Proposition \ref{Ricc}, the scalar curvature $\Scal_{g_{0}}\equiv\Scal(g_{0})$ is a homogeneous Laurent  polynomial of degree $-1$. Hence,  
\[
\Scal(c(t)g_{0})=c(t)^{-1}\Scal(g_{0})=\frac{1}{(1-2\lambda t)}\Scal(g_{0})\,.
\]
Consequently $\Scal(g(t))=\frac{1}{(1-2\lambda t)}\Scal(g_{0})>0$, for any $t\in(-\infty, \frac{1}{2\lambda})$, since    $\Scal(g_{0})>0$. Obviously, 
\[
\Scal'(g(t))=\frac{2\lambda\Scal(g_{0})}{(1-2\lambda t)^{2}}>0,
\]
 for any $(-\infty, \frac{1}{2\lambda})$ and when $t$ tends to $\frac{1}{2\lambda}$ from below, we see that $\Scal(g(t))\to+\infty$. Since $\Scal(g(t))$, as  a smooth function of $t$, is defined only  for $t\in(-\infty, \frac{1}{2\lambda})$, we conclude. Moreover, for  $t\to-\infty$ we get $\Scal(g(t))\to0$. \\
 (2) Similarly, by Proposition \ref{Ricc}, the components $\ric_{i}\equiv\ric_{i}(g_{0})=\ric_{i}^{0}$ $(1\leq i\leq r)$ of the Ricci tensor of $g_{0}$ are homogeneous Laurent polynomials of degree -1. Hence,
\[
\ric_{i}(c(t)g_{0})=\frac{1}{c(t)}\ric_{i}(g_{0})=\frac{1}{c(t)}\ric_{i}^{0}\,,
\]
i.e. $\ric_{i}(g(t))=c(t)^{-1}\ric_{i}^{0}$. The conclusion now easily  follows, since $g_{0}$ is Einstein and so $\ric_{i}^{0}=\lambda$, which is independent of $t$, for any  $1\leq i\leq r$.
\end{proof}
\begin{remark} $(i)$  The conclusions for the asymptotic behaviour of scalar curvature for the solutions $g(t)$ verify  a more general statement for the limit behaviour of the scalar curvature of homogeneous ancient solutions, obtained by Lafuente  \cite[Thm.~1.1, (i)]{Laf15}  in terms of the so-called  {\it bracket flow}.\footnote{The solutions  $g(t)$ described above have as maximal interval of definition the open set $(-\infty, \frac{1}{2\lambda_i})$,  so  the second part of \cite[Thm.~1.1]{Laf15} does not apply in our situation.} Later, for the convenience of the reader, in Section \ref{gloglo} we will illustrate Proposition \ref{ak}  by certain examples  (see Examples \ref{s2c}, \ref{3compexa}).\\
\noindent $(ii)$ Recall  by \cite[p.~545]{chow}   that $\Ric_{t}\equiv\Ric(g(t))=\Ric(g(0))=\lambda g_{0}=\frac{\lambda}{(1-2\lambda t)}g(t)$. Thus, the trivial solutions $g(t)$  given in Proposition \ref{ak}, being homothetic to invariant Einstein metrics, they also satisfy the Einstein equation (and therefore the Ricci flow equation), and hence $\lim_{t\to-\infty}\Ric_{t}=\lambda g_{0}=\lim_{t\to(\frac{1}{2\lambda})^{-}}\Ric_{t}$.  % Indeed,  
%\begin{eqnarray*}
%\Ric_{t}\equiv\Ric(g(t))&=&\sum_{i}x_{i}(t)\cdot\ric_{i}^{t}\cdot B|_{\fr{m}_i}=\sum_{i}c(t)x_{i}^{0}\cdot\frac{1}{c(t)}\ric_{i}^{0}\cdot B|_{\fr{m}_i}=\sum_{i}x_{i}^{0}\cdot\ric_{i}^{0}\cdot B|_{\fr{m}_i}\\
%&=&\Ric(g(0))=\lambda g_{0}=\frac{\lambda}{(1-2\lambda t)}g(t)\,,
%\end{eqnarray*}
%for any $t\in(-\infty, \frac{1}{2\lambda})$, and moreover  
%Moreover,  $\lim_{t\to-\infty}\Ric_{t}=\lambda g_{0}=\lim_{t\to(\frac{1}{2\lambda})^{-}}\Ric_{t}$.  I
The $\Ad(K)$-invariant $g(t)$-self-adjoint operator $\mathsf{r}_{g(t)}\equiv\mathsf{r}_{t} : \fr{m}\to\fr{m}$  (Ricci endomorphism) corresponding to the Ricci tensor $\Ric_{t}$ is defined by  $
 \Ric_{t}(X, Y)=g_{t}(\mathsf{r}_{t}(X), Y)$ for any $X, Y\in\fr{m}$. It satisfies the relation $\mathsf{r}_{t}=\frac{\lambda}{c(t)}A_{t}=\lambda A_{0}=\mathsf{r}_{0}$, for any  $t\in(-\infty, \frac{1}{2\lambda})$,
 where $A_{0}=\sum_{i=1}^{r}x_{i}^{0}\cdot\Id_{\fr{m}_i}$ is the  positive definite $\Ad(K)$-invariant $g_{0}$-self-adjoint operator  corresponding to the diagonal metric $g_{0}$. The  endomorphism of the 1-parameter family  $g(t)=c(t)g_{0}$, given by  $A_{t}=c(t)A_{0}$, is also  positive definite  for any $t\in(-\infty, \frac{1}{2\lambda})$, and the same satisfies $\mathsf{r}_{t}$. \\
\noindent $(iii)$ Obviously, Proposition \ref{ak} and the above conclusions can be extended to any homogeneous space $G/K$ of a compact semisimple Lie group $G$ modulo a compact subgroup $K\subset G$, with  a monotypic isotropy representation admitting an invariant Einstein metric $g_{0}$.  A simple  example is given below.  On the other side, there are examples of effective compact homogeneous spaces,  with non-monotypic isotropy representation, i.e. $\fr{m}_{i}\cong\fr{m}_{j}$ for some $1\leq i\neq j\leq r$, for which the invariant (Einstein) metrics are still diagonal. To take a taste, consider the Stiefel manifold $V_{2}(\R^{\ell+1})=G/K=\SO(\ell+1)/\SO(\ell-1)$ and assume for simplicity that  $\ell\neq 3$.  This is a compact homogeneous space, admitting  a $\U(1)$-fibration over the Grassmannian  ${\rm Gr}_{2}^{+}(\R^{\ell+1})=\SO(\ell+1)/\SO(\ell-1)\times\SO(2)$. Let $\fr{so}(\ell+1)=\fr{so}(\ell-1)\oplus\fr{m}$ be  a $B$-orthogonal reductive decomposition. Then, it is not hard to see that $\fr{m}=\fr{m}_0\oplus\fr{m}_1\oplus\fr{m}_2$, where $\fr{m}_{0}$ is 1-dimensional and $\fr{m}_1\cong\fr{m}_2$ are two irreducible  submodules of dimension $\ell-1$, both isomorphic to the standard representation of $\SO(\ell-1)$.  Hence, the isotropy representation of $V_{2}(\R^{\ell+1})$ is {\it not} monotypic.  However, the invariant metrics on  $V_{2}(\R^{\ell+1})$ can be shown that are still diagonal. 	This is based on the action of the {\it generalized Weyl group} (gauge group) $N_{G}(K)/K$ on the space $P(\fr{m})^{\Ad(K)}$ of $\Ad(K)$-invariant inner products on $\fr{m}$ (see for example \cite{NRS, marcus}). For the specific case of  $V_{2}(\R^{\ell+1})$, the group $N_{G}(K)/K$ is isomorphic to a circle and   this action was used in \cite[p.~121]{Kerr98} to eliminate the off-diagonal components of the invariant metrics  (note that for $\ell\neq 3$, $V_{2}(\R^{\ell+1})$ admits a unique $\SO(2\ell+1)$-invariant Einstein metric).  Hence, the results discussed  above can also be extended in that more general case.
\end{remark}
\begin{example}\label{isoire}
Let $M=G/K$ be an isotropy irreducible homogeneous space of a compact simple Lie group $G$. Consider a $B$-orthogonal reductive decomposition $\fr{g}=\fr{k}\oplus\fr{m}$.  Then, $\M^{G}=\R_{+}$ and the Killing form $g_{B}=B|_{\fr{m}}$
is the unique invariant Einstein metric (up to a scalar). Hence, $g(t)=(1-2\lambda_{B}t)g_{B}$ is a (trivial) homogeneous ancient solution of the corresponding homogeneous Ricci flow, defined on the open interval $(-\infty, \frac{1}{2\lambda_B})$, where $\lambda_{B}>0$ is the Einstein constant of $g_{B}$. This applies in particular to any symmetric flag manifold $M=G/K$ of a compact simple Lie group $G$ (i.e. a compact isotropy irreducible HSS where $r=1$).
\end{example}
\begin{definition}A homogeneous ancient solution of (\ref{homflow1}) is called {\it non-collapsed},  if the corresponding curvature
normalized metrics have a uniform lower injectivity radius bound.
\end{definition}

Non-collapsed homogeneous ancient solutions of the Ricci flow on compact homogeneous space  have been recently studied in \cite{BLM17}. In this work, among other results,   the authors  proved that:
\begin{itemize}
\item  $(\upalpha)$  If $G, K$ are connected and does not exist some intermediate group $K\subset L\subset G$ such that $L/K$ is a torus, i.e. if $G/K$ is {\it not} a homogeneous torus bundle $G/K\to G/L$ over $G/L$,  then any ancient solution on $G/K$ is non-collapsed (\cite[Rem~5.3]{BLM17}).
 %Non-collapsed ancient solutions on a compact homogeneous
%space $M=G/K$  emanate from a $G$-invariant Einstein metric on $G/K$.
\item $(\upbeta)$ Non-trivial  homogeneous ancient solutions of Ricci flow must develop a {\it Type I singularity}
close to their extinction time  and also to the past (see \cite{BLM17},  Corollary 2, page 2, and pages 24-25).
\end{itemize}
 Hence we deduce that
\begin{corol}\label{noncollapsed}% \textnormal{(\cite{BoLaf})}\
Let $M=G/K$ be a flag manifold as in Proposition \ref{ak}.   Then any non-trivial ancient solution of the homogeneous Ricci flow, if it exists,  emanates from an invariant Einstein metric, is non-collapsed and develops a Type I singularity close to the  extinction time (and also to the past). In particular, the trivial ancient solutions of the form $g(t)=(1-2\lambda t)g_0$, where $g_0 \in \mathcal{M}^G$ is an invariant Einstein metric on $M$ with Einstein constant $\lambda$, are non-collapsed and as $t\rightarrow \frac{1}{2\lambda}$ the volume of $M=G/K$  with respect to $g(t)$ tends to $0$, 
\[
Vol_{g(t)}(G/K)\rightarrow 0,
\]
 i.e., $M=G/K$ shrinks to a point in finite time.
\end{corol}
\begin{proof}
For the record, let us assume that there exists some intermediate closed subgroup $K\subset L\subset G$  such that $L/K$ is torus $T^{s}$ for some $s\geq 1$. Then obviously, it must be $\rnk L>\rnk K$. But $\rnk K=\rnk G$ and  then the inclusion  $L\subset G$ gives a contradiction.  Hence, $M=G/K$ cannot be homogeneous torus bundle.  Thus,  according to  $(\upalpha)$ any  possible homogeneous ancient solution  of (\ref{homflow1}) on $M=G/K$ must be non-collapsed. The other claims and the assertions for Type I behaviour, follow now by $(\upbeta)$.
\end{proof}

\begin{remark}
Recall that on a compact homogeneous space $M=G/K$ the total scalar curvature functional 
\[
\bold{S}(g)=\displaystyle\int_{M}\Scal(g)dV_{g},
\]
 restricted on the  set $\M_1^{G}$ of $G$-invariant metrics of volume 1, coincides with the smooth function $\M^{G}_{1}\ni g\longmapsto\Scal(g)\equiv\Scal_{g}$, since the scalar curvature $\Scal_{g}$ of $g$ is a constant function on $M$.  In this case, $G$-invariant Einstein metrics of volume 1 on $G/K$  are precisely   the   critical points of the restriction $\bold{S}|_{\M^G_1}$.  A $G$-invariant Einstein  metric is called {\it unstable} if is not a {\it local maximum} of $\bold{S}|_{\M^G_1}$. 
By \cite[Lem.~5.4]{BLM17} it also known that {\it for any unstable homogeneous Einstein metric  on a compact homogeneous
space $G/K$, there exists a non-collapsed invariant ancient solution emanating from it}. For general flag manifolds, up to our knowledge,  is an open question  if all  possible existent invariant Einstein metrics are unstable or not.
However, by  \cite[Thm.~1.2]{Chry211} it is known that on a flag manifold with $r=2$, i.e. with two isotropy summands, there exist two non-isometric invariant Einstein metrics which are both local minima of $\bold{S}|_{\M^G_1}$, and hence unstable in the above sense. Thus, for  $r=2$ the existence of non-collapsed ancient solutions of  system (\ref{homflow1}) can be also obtained by a direct combination of the results in \cite{Chry211, BLM17}.
\end{remark}

%%%%%%%%%%%%%%%%%%%%%%%%%%%%SUBECTION 2.2  POINCARE COMPACT%%%%%%%%%%%%%%%%%%
 %%%%%%%%%%%%%%%%%%%%%%%%%%%%%%%%%%%%%%%%%%%%%%%%%%%%%%%%%%%%%
 %%%%%%%%%%%%%%%%%%%%%%%%%%%%%%%%%%%%%%%%%%%%%%%%%%%%%%%%%%%%

\subsection{The Poincar\'{e} compactification procedure}  \label{poinc}
To study the  {\it asymptotic behaviour} of  homogeneous Ricci flow solutions,  even of more general abstract solutions than these given in Proposition \ref{ak},   one  can successfully use the  compactification method of Poincar\'{e} which we refresh below, adapted to our setting (see also \cite{Grama09, AC11, Grama12}). Indeed,  the Poincar\'{e} compactification procedure is used to study the behaviour of a polynomial system of ordinary differential equations in a neighbourhood of infinity, see \cite{Gon, Vidal} for an explicit description.  We describe it here, in a form suitable for our purposes.  

Let $(M=G/K, g)$ be a flag manifold with $\M^{G}\cong\R^r_{+}$ for some $r\geq  1$.   By using the expressions of Proposition \ref{Ricc} and    after multiplying the right-hand side of the equations in (\ref{homflow1}),   with a suitable  {\it positive factor},  we obtain a qualitative equivalent  dynamical system consisting of homogeneous polynomial   equations of positive degree.  Let us denote this  system obtained from  the homogeneous Ricci flow via  this procedure, by
\begin{equation}\label{PRF}
\Big\{\dot{x}_k=\RF_k(x_1,\ldots, x_r) :  k=1,\ldots, r\Big\}.
\end{equation}  
So, for any  $1\leq k\leq r$, $\RF_k(x_1,..,x_r)$ are homogeneous polynomials, whose {\it maximum degree is defined to be the degree $d$ of the system {\rm(\ref{PRF})}}. Let us proceed with the following definition.
\begin{definition}
The vector field associated to the system (\ref{PRF}) will be denoted by $
X(x_1,..,x_r):=\big(\RF_1(x_1,\ldots, x_r),\ldots,  \RF_r(x_1,\dots, x_r)\big)$ and referred to as the {\it homogeneous vector field associated to the homogeneous Ricci flow on $(M=G/K, g)$.} 
\end{definition}
Consider  the subset of the unit sphere $\Ss^{r}\subset\R^{r+1}$, containing all points having non-negative coordinates, i.e.
$\Ss^r_{\geq}=\Big\{y\in \R^{r+1} :  \|y\| =1,\ y_i\geq 0,\ i=1,\ldots, r+1\Big\}\subset\Ss^r$.
It  is also convenient to identify $\R^{r}_{\geq}$ with the subset of $\R^{r+1}$ defined by 
\[
\mathsf{T}:=\Big\{(y_1,\ldots, y_r, y_{r+1})\in\R^{r+1} : y_{r+1}=1, \ y_i\geq 0,\ \forall \ i=1, \ldots, r\Big\}\cong\R^{r}_{\geq}\,.
\]
Consider  now the  {\it central projection} $
f :  \R^{r}_{\geq}\cong\mathsf{T} \rightarrow \Ss^r_{\geq}$,
assigning to every $p\in\mathsf{T}$  the point $\mathsf{T}(p)\in \Ss^r_{\geq}$, defined as follows: $\mathsf{T}(p)$ is the intersection of the straight line joining the initial point $p$ with the origin of $\R^{r+1}$. The explicit form of $f$ is given by  
\[
f(y_1,\ldots, y_r,1)=\frac{1}{\|(y_1,\ldots,y_r,1)\|}(y_1,\ldots, y_r,1).
\]
Through this projection, $\R^{r}_{\geq}$ can be identified with the subset of $\Ss^r_{\geq}$ with $y_{i+1}> 0$. Moreover,  the equator of $\Ss^r_{\geq}$, that is  
\[
\Ss^{r-1}_{\geq}=\Big\{y\in \Ss^{r}_{\geq} :  y_{r+1}=0\Big\},
\]
 is identified as the {\it infinity of $\R^r_{\geq}$}. To be more explicit, let us state this as a definition.
 \begin{definition}\label{infin}
A {\it point  at infinity} of $\M^G\cong\mathbb{R}^r_{\geq}$ is understood  to  be point of the equator $\Ss^{r-1}_{\geq}$.
\end{definition}
Note that $\Ss^r_{\geq}$ is diffeomorphic with the standard $r$--dimensional simplex
\[
\Delta^{r}=\Big\{(y_1,\ldots, y_{r+1})\in \R^{r+1} :  \sum_{i}y_i=1,\  \forall \ y_i\geq 0,\ i=1,..,r+1\Big\}
\]
and thus with the so-called {\it non-negative part of the real projective space} $\RP^{r}_{\geq}=(\R^{r+1}_{\geq}\backslash\{0\})/\R_{+}$ (see \cite[Ch.~4]{Ful}). Via  push-forward, the central projection carries the vector field $X$ onto $\Ss^r_{\geq}$.  The vector field  obtained by  this procedure, i.e. the vector field
\[
p(X)(y):=y_{r+1}^{d-1}f_*X(y)
\]
is  called the {\it Poincar\'{e} compactification of   $X$}, and this  is an analytic vector field defined on all $\Ss^r_{\geq}$.
Actually, in order to perform computations with $p(X)$ one needs its expressions in a chart. Thus, consider the chart
\[
U=\Big\{y\in \Ss^r : y_1>0\Big\}
\]
of $\Ss^r$, and project it to the plane $\Big\{(y_1,\ldots, y_{r+1})\in \mathbb{R}^{r+1} :  y_1=1)\Big\}\cong\R^r_{\geq}$. Then, via  the corresponding central projection, assign  to every point of $U$ the point of intersection of the straight line joining the origin with the original point.  This second central projection   denoted by $F : U\rightarrow \R^r_{\geq }$, has obviously the form $F(y_1, y_2,\ldots, y_{r+1})=(1,\frac{y_2}{y_1},\ldots, \frac{y_{r+1}}{y_1})$. 
We can now compute the local expression of $p(X)$ in the chart $U$; As above, let us denote by $x_i$ the coordinates on $\mathbb{R}^r_{\geq}$. Then, we obtain
\begin{prop}\label{lochf}
The local expression of the Poincar\'{e} compactification $p(X)$ of the homogeneous vector field   $X$ associated to the Ricci flow on $(M=G/K, g)$, reads as
\begin{equation}\label{locexp}
\dot{x}_i = x_{r}^d(-x_i\overline{\RF}_1+\overline{\RF}_{i+1}) \ \ \text{for any} \ \  i=1,\ldots, r-1\,,\ \ \text{and} \  \
\dot{x}_{r} = x_{r}^d(-x_{r}\overline{\RF}_1)\,,
\end{equation}
where  $\overline{\RF}_i(x_1,\ldots, x_{r}):=\displaystyle\RF_i(\frac{1}{x_{r}},\frac{x_2}{x_{r}},\ldots, \frac{x_{r-1}}{x_{r}})$.\footnote{Here, we have omitted the term $\frac{1}{\|(x_1,\ldots,x_r,1)\|}$,
since it does not affect  the qualitative behaviour of the system.} 
\end{prop}
\begin{remark}\label{fixinfinity}
In the expressions given by  (\ref{locexp}),  the factor $x_{r}^{d}$ is canceled by the dominators of the rational polynomials $\overline{\RF}_{i}(x_1,\ldots, x_{r})$. After doing so and by locating the fixed points of the resulting dynamical system, under the condition $x_{r}=0$,  we obtain the so-called   {\it fixed points  at the infinity of $\M^G\cong\R^{r}_{+}$ of the homogeneous Ricci flow on $(M=G/K, g)$.}
\end{remark}
\begin{example}\label{inf2d}
For $r=2$ the expression of $p(X)$ in the local chart $U$ is given by
\[
\dot{x}_1=x_2^d(-x_1\overline{\RF}_1+\overline{\RF}_2)\,,\quad
\dot{x}_2=x_2^d(-x_2\overline{\RF}_1)\,,
\]
where $\overline{\RF}_i(x_1, x_2)=\RF_i(\frac{1}{x_2},\frac{x_1}{x_2})$ for any $i=1, 2$. We may omit  the term $\frac{1}{\|(x_1,x_2,1)\|}$,  by using a time reparametrization. Fixed points of the homogeneous Ricci flow on $(M=G/K, g)$ at infinity of $\M^G=\R^{2}_{+}$, can be studied by setting $x_2=0$.
For $r=3$ the expression of $p(X)$ in the local chart $U$ reads by
\[
\dot{x}_1=x_3^d(-x_1\overline{\RF}_1+\overline{\RF}_2)\,,\quad
\dot{x}_2=x_3^d(-x_2\overline{\RF}_1+\overline{\RF}_3)\,,\quad
\dot{x}_3=x_3^d(-x_3\overline{\RF}_1)\,,
\]
where $\overline{\RF}_i=\RF_i(\frac{1}{x_3},\frac{x_1}{x_3},\frac{x_2}{x_3})$, for any $i=1, 2, 3$.  %Again one can omit the term $\frac{1}{\|(x_1,x_2,x_3,1)\|}$, by using a time reparametrization. 
 In an analogous way, fixed points at infinity  can be studied by setting $x_3=0$, while similarly are treated cases with $r>3$.
\end{example}
 
%%%%%%%%%%%%%%%%%%%%%%%%%%%%SECTION 3%%%%%%%%%%%%%%%%%%
 %%%%%%%%%%%%%%%%%%%%%%%%%%%%%%%%%%%%%%%%%%%%%%%%%%%%%%%%%%%%%
 %%%%%%%%%%%%%%%%%%%%%%%%%%%%%%%%%%%%%%%%%%%%%%%%%%%%%%%%%%%%
\section{Global study of HRF on flag spaces $M=G/K$ with $b_{2}(M)=1$}\label{gloglo}
We turn  now our attention to    the  global study of the classical Ricci flow equation for an initial invariant metric on (non symmetric)  flag manifolds $M=G/K$ with $b_{2}(M)=1$. Again we can work with $G$ simple. Let us begin first with a few details  about this specific class of flag spaces.

 %%%%%%%%%%%%%%%%%%%%%%%%%%%%SUBECTION 2.1 FLAG SPACES b2=1 %%%%%%%%%%%%%%%%%%
 %%%%%%%%%%%%%%%%%%%%%%%%%%%%%%%%%%%%%%%%%%%%%%%%%%%%%%%%%%%%%

\subsection{Flag manifolds $M=G/K$ with $b_{2}(M)=1$}
According to Proposition \ref{betti},   flag manifolds $M=G/K$ of a compact simple Lie group $G$ with $b_{2}(M)=1$ are defined by painting black in the Dynkin diagram of $G$ only a simple root, i.e. $\Pi_{M}=\{a_{i_{o}}\}$ for some $a_{i_{o}}\in\Pi$.
The number of the isotropy summands of a flag space $M=G/K$ with $b_{2}(M)=1$ can be read from the PDD, at least  when we encode with it the so-called   {\it Dynkin marks} of  the simple roots; These  are the positive integers  coefficients appearing in the expression of the highest root of $G$ as a linear combination of    simple roots.  For flag manifolds with $b_{2}(M)=1$, i.e. $\Pi_{M}=\Pi/\Pi_{K}=\{a_{i_{o}} : 1\leq i_{o}\leq \ell=\rnk G\}$  we have the relation $r={\rm Dynk}(\al_{i_{o}})$ (see \cite{Chry212, CS1}),
where $r$ is the integer appearing in (\ref{mmm}). In other words, a flag manifold $M=G/K$ with $b_{2}(M)=1$ and $r$ isotropy summands is obtained    by painting black  a simple root $\al_{i_{o}}$  with Dynkin mark $r$, and conversely. Since for a compact simple Lie group $G$ the maximal Dynkin mark equals to 6 (and occurs for $G=\E_8$ only), we result with the bound $1\leq r\leq 6$. %By Lemma \ref{monotypic} any $G$-invariant metric $g\in\M^{G}$   is determined by $r$ positive real variables, i.e. $
%g=\sum_{i=1}^{r}x_{i}\cdot B|_{\fr{m}_i}$  with $x_i\in\R_{+}$ for any $i=1, \ldots, r$. The Ricci tensor $\Ric_{g}$ of $g$ has the form $
%\Ric_{g}=\sum_{i=1}^{r}(x_{i}\cdot\ric_{i})\cdot B|_{\fr{m}_{i}}$. 

So, from now on assume that $M=G/K$ is a flag manifold as above with $b_2(M)=1$ and $2\leq r\leq 6$. The classification of such flag manifolds can be found in  \cite{CS1};  There, in combination with earlier results of Kimura, Arvanitoyeorgos and the second author \cite{Kim90, Chry210, Chry211}  is proved that any such space admits a finite number of non-isometric (non-K\"ahler) invariant Einstein metrics, a result which supports the still open    {\it finiteness conjecture} of B\"ohm-Wang-Ziller (\cite{BWZ04}).   Note    that  $M=G/K$ admits a unique invariant complex structure, and thus a  unique invariant K\"ahler-Einstein metric, with explicit form  (see \cite{Bor, CS1})
\begin{equation}\label{KE}
g_{KE}=\sum_{i}^{r}i\cdot B|_{\fr{m}_i}=B|_{\fr{m}_1}+2B|_{\fr{m}_2}+\ldots+rB|_{\fr{m}_r}\,.
\end{equation}
 The isotropy summands satisfy the relations
\[
 [\fr{k}, \fr{m}_{i}]\subset\fr{m}_{i}, \quad  [{\frak m}_i, {\frak m}_i]\subset{\frak k} + {\frak m}_{ 2 i }, \quad [{\frak m}_i, {\frak m}_j]\subset {\frak m}_{ i + j }+ {\frak m}_{| i - j |}  \ \ (i\neq j). 
\]
Hence, the non-zero structure constants are listed as follows:
\[
\begin{tabular} {r || l | l}
$r$   & non-zero structure constants $c_{ij}^{k}$   \\
\hline
2 &   $c_{11}^{2}$  (\cite{Chry211}) \\
3 & $c_{11}^{2}$, $ c_{12}^{3}$  (\cite{Kim90, AC11})\\
4 & $c_{11}^{2}$, $ c_{12}^{3}$, $c_{13}^{4}$, $ c_{22}^{4}$  (\cite{Chry210})\\
5 & $c_{11}^{2}$, $ c_{12}^{3}$, $c_{13}^{4}$, $c_{14}^{5}$, $c_{22}^{4}$, $c_{23}^{5}$  (\cite{CS1})\\
6 & $c_{11}^{2}$, $ c_{12}^{3}$, $c_{13}^{4}$, $c_{14}^{5}$, $c_{15}^{6}$, $c_{22}^{4}$, $c_{23}^{5}$, $c_{24}^{6}$, $c_{33}^{6}$    (\cite{CS1})\\
\end{tabular}
\]
For $r=2, 3$, the use of the K\"ahler-Einstein metric is sufficient for an explicit computation of $c_{ij}^{k}$, and this yields  a general expression of them in terms of $d_{i}=\dim\fr{m}_i$ (see \cite{Kim90, Chry211, AC11}). For $4\leq r\leq 6$ more advanced techniques are necessary for the computation of $c_{ij}^{k}$, which depend on the specific coset (see \cite{Chry210, CS1}). \\
\noindent {{\bf Case}  $\bold{r=2}$:} All flag manifolds $M=G/K$ with $G$ simple  and $r=2$ have $b_{2}(M)=1$, see  \cite{Chry211} for details.  We recall that:
\begin{eqnarray}
\ric_{1}&=&\displaystyle\frac{1}{2x_{1}} - c_{11}^{2}\displaystyle\frac{x_{2}}{{2d_{1}}{x_{1}^2}}\,, \quad
\ric_{2}=\displaystyle\frac{1}{2x_{2}} +  \frac{c_{11}^{2}}{4d_{2}}\Big(\displaystyle\frac{x_{2}}{{x_{1}^2}}  -\displaystyle\frac{2}{x_{2}}\Big)\,,\quad c_{11}^{2}=\displaystyle\frac{d_{1}d_{2}}{d_{1}+4d_{2}}\,,\nonumber\\
\Scal_{g}&=&\sum_{i=1}^{2}d_{i}\cdot \ric_{i}=\displaystyle\frac{1}{2}\Big(\frac{d_{1}}{x_{1}}+\frac{d_{2}}{x_{2}}\Big)-\frac{c_{11}^{2}}{4}\Big(\frac{x_{2}}{x_{1}^{2}}+\frac{2}{x_{2}}\Big).\label{sca2c}
\end{eqnarray}
\noindent {{\bf Case}  $\bold{r=3}$:}   Let $M=G/K$ be a flag manifold with $b_{2}(M)=1$ and $r=3$.  Such  flag spaces have been classified by \cite{Kim90}, see also \cite{AC11, Grama20}.  They all   correspond to exceptional compact simple Lie groups, but the correspondence is not a bijection. One should mention that not  all flag   manifolds with $r=3$ are  exhausted by this type of homogeneous spaces; This means that still one can construct  flag spaces with  $r=3$ and $b_{2}(M)=2$. We recall that
\begin{eqnarray}
 c_{11}^{2}&=&\frac{d_{1}d_{2} + 2d_{1}d_{3} - d_{2}d_{3}}{d_{1} + 4d_{2} + 9d_{3}}\,,    \quad c_{12}^{3}=\frac{d_{3}(d_{1} + d_{2})}{d_{1} + 4d_{2} + 9d_{3}}\,,\nonumber\\
  \ric_{1}&=&\displaystyle\frac{1}{2x_{1}} - \displaystyle\frac{c_{11}^{2}x_{2}}{2d_{1}x_{1}^{2}}+\displaystyle\frac{c_{12}^{3}}{2d_{1}}\Big(\frac{x_{1}}{x_{2}x_{3}}-\frac{x_{2}}{x_{1}x_{3}}-\frac{x_{3}}{x_{1}x_{2}}\Big),\nonumber\\
\ric_{2}&=&\displaystyle\frac{1}{2x_{2}} +  \displaystyle\frac{c_{11}^{2}}{4d_{2}}\Big(\frac{x_{2}}{x_{1}^2} -\frac{2}{x_{2}}\Big)+\displaystyle\frac{c_{12}^{3}}{2d_{2}}\Big(\frac{x_{2}}{x_{1}x_{3}}-\frac{x_{1}}{x_{2}x_{3}}-\frac{x_{3}}{x_{1}x_{2}}\Big),\nonumber\\
\ric_{3}&=&\displaystyle\frac{1}{2x_{3}} +  \displaystyle\frac{c_{12}^{3}}{2d_{3}}\Big(\frac{x_{3}}{x_{1}x_{2}}-\frac{x_{1}}{x_{2}x_{3}}-\frac{x_{2}}{x_{1}x_{3}}\Big),\nonumber\\
\Scal_{g}&=&\frac{1}{2}\Big(\frac{d_{1}}{x_{1}}+\frac{d_{2}}{x_{2}}+\frac{d_{3}}{x_{3}}\Big)-\displaystyle\frac{c_{11}^{2}}{4}\Big(\frac{x_{2}}{x_{1}^{2}}+ \frac{2}{x_{2}}\Big) -  \displaystyle\frac{c_{12}^{3}}{2}\Big(\frac{x_{1}}{x_{2}x_{3}}+\frac{x_{2}}{x_{1}x_{3}}+\frac{x_{3}}{x_{1}x_{2}}\Big). \label{sca3c}
\end{eqnarray}
\noindent {{\bf Case}  $\bold{r=4}$:} Let $M=G/K$ be a flag manifold with $G$ simple, $r=4$ and $b_{2}(M)=1$.  There  exist   four such   homogeneous spaces and all of the them correspond to an exceptional  Lie group. To  give the reader a small taste of PDDs, we present these flag spaces below, together with the corresponding   PDD and Dynkin marks. As for the case $r=3$, one should be aware that  still exist flag manifold with $r=4$, but $b_{2}(M)=2$.
{\small{
\begin{center}\[ \begin{tabular}{l | c}
$M=G/K$ with $b_{2}(M)=1$ and $r=4$ & PDD \\
\thickline
$\F_4/\SU(3)\times\SU(2)\times\U(1)$ & $\begin{picture}(160,25)(38, 0)
\put(87,10){\circle{4 }}
\put(87,18){\makebox(0,0){$\al_1$}}
\put(87,2){\makebox(0,0){2}}
\put(89,10.5){\line(1,0){14}}
\put(105,10){\circle{4 }}
\put(105,18){\makebox(0,0){$\al_2$}}
\put(105,2){\makebox(0,0){3}}
\put(107, 11.3){\line(1,0){12.4}}
\put(107, 9.1){\line(1,0){12.6}}
\put(117.6, 8){\scriptsize $>$}
\put(124.5,10){\circle*{4 }}
\put(124.5, 18){\makebox(0,0){$\al_3$}}
\put(124.5, 2){\makebox(0,0){4}}
\put(126,10.5){\line(1,0){16}}
\put(144,10){\circle{4 }}
\put(144,18){\makebox(0,0){$\al_4$}}
\put(144,2){\makebox(0,0){2}}
\end{picture}$ \\
\hline
$\E_7/\SU(4)\times\SU(3)\times\SU(2)\times\U(1)$ & 
$\begin{picture}(160,25)(-25, 0)

\put(15, 9.5){\circle{4 }}
\put(15, 18){\makebox(0,0){$\al_1$}}
\put(15,2){\makebox(0,0){1}}
\put(17, 10){\line(1,0){14}}
\put(33.5, 9.5){\circle{4 }}
\put(33.5, 18){\makebox(0,0){$\al_2$}}
\put(33.5,2){\makebox(0,0){2}}
\put(35, 10){\line(1,0){13.6}}
 \put(51, 9.5){\circle{4 }}
 \put(51, 18){\makebox(0,0){$\al_3$}}
\put(51,2){\makebox(0,0){3}}
\put(69,-6){\line(0,1){14}}
\put(53,10){\line(1,0){14}}
\put(69,9.5){\circle*{4 }}
\put(69.9, 18){\makebox(0,0){$\al_4$}}
\put(65,2){\makebox(0,0){4}}
\put(69,-8){\circle{4}}
\put(79, -10){\makebox(0,0){$\al_7$}}
\put(62,-10){\makebox(0,0){2}}
\put(71,10){\line(1,0){16}}
\put(89,9.5){\circle{4 }}
\put(89, 18){\makebox(0,0){$\al_5$}}
\put(89,2){\makebox(0,0){3}}
\put(90.7,10){\line(1,0){16}}
\put(109,9.5){\circle{4 }}
\put(109, 18){\makebox(0,0){$\al_6$}}
\put(109, 2){\makebox(0,0){2}}
\end{picture}
$  \\\\
\hline
$\E_8/\SO(10)\times\SU(3)\times\U(1)$  & 
$\begin{picture}(160,25)(-5, 0)
\put(15, 9.5){\circle{4 }}
\put(15, 18){\makebox(0,0){$\al_1$}}
\put(15,2){\makebox(0,0){2}}
\put(17, 10){\line(1,0){14}}
\put(33.5, 9.5){\circle{4 }}
\put(33.5, 18){\makebox(0,0){$\al_2$}}
\put(33.5,2){\makebox(0,0){3}}
\put(35, 10){\line(1,0){13.6}}
 \put(51, 9.5){\circle*{4 }}
 \put(51, 18){\makebox(0,0){$\al_3$}}
\put(51,2){\makebox(0,0){4}}
\put(53,10){\line(1,0){14}}
\put(69,9.5){\circle{4 }}
\put(69, 18){\makebox(0,0){$\al_4$}}
\put(69,2){\makebox(0,0){5}}
\put(89,-8){\circle{4}}
\put(99, -9.5){\makebox(0,0){$\al_8$}}
\put(82,-9.5){\makebox(0,0){3}}
\put(89,-6){\line(0,1){14}}
\put(71,10){\line(1,0){16}}
\put(89,9.5){\circle{4 }}
\put(89, 18){\makebox(0,0){$\al_5$}}
\put(84,2){\makebox(0,0){6}}
\put(90.7,10){\line(1,0){16}}
\put(109,9.5){\circle{4 }}
\put(109, 18){\makebox(0,0){$\al_6$}}
\put(109,2){\makebox(0,0){4}}
\put(111,10){\line(1,0){16}}
\put(129,9.5){\circle{4 }}
\put(129, 18){\makebox(0,0){$\al_7$}}
\put(129,2){\makebox(0,0){2}}
\end{picture}
$ \\\\
\hline
$\E_8/\SU(7)\times\SU(2)\times\U(1)$  & 
$\begin{picture}(160,25)(-5, 0)
\put(15, 9.5){\circle{4 }}
\put(15, 18){\makebox(0,0){$\al_1$}}
\put(15,2){\makebox(0,0){2}}
\put(17, 10){\line(1,0){14}}
\put(33.5, 9.5){\circle{4 }}
\put(33.5, 18){\makebox(0,0){$\al_2$}}
\put(33.5,2){\makebox(0,0){3}}
\put(35, 10){\line(1,0){13.6}}
 \put(51, 9.5){\circle{4 }}
 \put(51, 18){\makebox(0,0){$\al_3$}}
\put(51,2){\makebox(0,0){4}}
\put(53,10){\line(1,0){14}}
\put(69,9.5){\circle{4 }}
\put(69, 18){\makebox(0,0){$\al_4$}}
\put(69,2){\makebox(0,0){5}}
\put(89,-8){\circle{4}}
\put(99, -9.5){\makebox(0,0){$\al_8$}}
\put(82,-9.5){\makebox(0,0){3}}
\put(89,-6){\line(0,1){14}}
\put(71,10){\line(1,0){16}}
\put(89,9.5){\circle{4 }}
\put(89, 18){\makebox(0,0){$\al_5$}}
\put(84,2){\makebox(0,0){6}}
\put(90.7,10){\line(1,0){16}}
\put(109,9.5){\circle*{4 }}
\put(109, 18){\makebox(0,0){$\al_6$}}
\put(109,2){\makebox(0,0){4}}
\put(111,10){\line(1,0){16}}
\put(129,9.5){\circle{4 }}
\put(129, 18){\makebox(0,0){$\al_7$}}
\put(129,2){\makebox(0,0){2}}
\end{picture} 
$\\\\
\hline
\end{tabular}
\]
\end{center}}}
We also recall that
\begin{eqnarray*}
\ric_1&=&\displaystyle\frac{1}{2x_{1}}-\frac{c_{11}^{2}}{2d_{1}}\frac{x_{2}}{x_{1}^2}+
\frac{c_{12}^{3}}{2d_{1}}\Big(\frac{x_{1}}{x_{2}x_{3}}-\frac{x_{2}}{x_{1}x_{3}}-\frac{x_{3}}{x_{1}x_{2}}\Big)+
\frac{c_{13}^{4}}{2d_{1}}\Big(\frac{x_{1}}{x_{3}x_{4}}-\frac{x_{3}}{x_{1}x_{4}}-\frac{x_{4}}{x_{1}x_{3}}\Big)\,, \\
\ric_2&=&\displaystyle\frac{1}{2x_{2}}-\frac{c_{22}^{4}}{2d_{2}}\frac{x_{4}}{x_{2}^2}+
\frac{c_{11}^{2}}{4d_{2}}\Big(\frac{x_{2}}{x_{1}^2}-\frac{2}{x_{2}}\Big)+
\frac{c_{12}^{3}}{2d_{2}}\Big(\frac{x_{2}}{x_{1}x_{3}}-\frac{x_{1}}{x_{2}x_{3}}-\frac{x_{3}}{x_{1}x_{2}}\Big)\,, \\
\ric_3&=&\displaystyle\frac{1}{2x_{3}}+
\frac{c_{12}^{3}}{2d_{3}}\Big(\frac{x_{3}}{x_{1}x_{2}}-\frac{x_{2}}{x_{1}x_{3}}-\frac{x_{1}}{x_{2}x_{3}}\Big)+
\frac{c_{13}^{4}}{2d_{3}}\Big(\frac{x_{3}}{x_{1}x_{4}}-\frac{x_{1}}{x_{3}x_{4}}-\frac{x_{4}}{x_{1}x_{3}}\Big)\,,\\
\ric_4&=&\displaystyle\frac{1}{2x_{4}}+
\frac{c_{22}^{4}}{4d_{4}}\Big(\frac{x_{4}}{x_{2}^2}-\frac{2}{x_{4}}\Big)+
\frac{c_{13}^{4}}{2d_{4}}\Big(\frac{x_{4}}{x_{1}x_{3}}-\frac{x_{1}}{x_{3}x_{4}}-\frac{x_{3}}{x_{1}x_{4}}\Big)\,,\\
  \Scal_g&=&\frac{1}{2}\sum_{i=1}^{4}\frac{d_{i}}{x_{i}} - \frac{c_{12}^{3}}{2}(\frac{x_{1}}{x_{2}x_{3}} + \frac{x_{2}}{x_{1}x_{3}} + \frac{x_{3}}{x_{1}x_{2}}) - 
 \frac{c_{13}^{4}}{2}(\frac{x_{1}}{x_{3}x_{4}} + \frac{x_{3}}{x_{1}x_{4}}+\frac{x_{4}}{x_{1}x_{3}}) \\ 
 && -\frac{c_{11}^{2}}{4}(\frac{x_{2}}{x_{1}^{2}} + \frac{2}{x_{2}})-\frac{c_{22}^{4}}{4}(\frac{x_{4}}{x_{2}^{2}} + \frac{2}{x_{4}})\,.
\end{eqnarray*}
Let us finally present the values of $c_{ij}^{k}$ and the corresponding dimensions:
{\small\[
\begin{tabular}{l|l|l|l|l}
$M=G/K$ & $c_{22}^{4}$ & $c_{11}^{2}$ & $c_{12}^{3}$ & $c_{13}^{4}$\\
\thickline 
$\F_4/\SU(3)\times\SU(2)\times\U(1)$ & $2$ & $2$ & $1$ & $2/3$\\
\hline
$\E_7/\SU(4)\times\SU(3)\times\SU(2)\times\U(1)$  & $2$ & $8$ & $4$ & $4/3$\\
\hline 
$\E_8/\SO(10)\times\SU(3)\times\U(1)$ & $2$ & $16$ & $8$ & $8/5$\\
\hline
$\E_8/\SU(7)\times\SU(2)\times\U(1)$ & $14/3$ & $14$ & $7$ & $14/5$\\
\thickline
 & $d_{1}$ & $d_{2}$ & $d_{3}$ & $d_{4}$\\
\thickline 
$\F_4/\SU(3)\times\SU(2)\times\U(1)$ & $12$ & $18$ & $4$ & $6$\\
\hline
$\E_7/\SU(4)\times\SU(3)\times\SU(2)\times\U(1)$  & $48$ & $36$ & $16$ & $6$\\
\hline 
$\E_8/\SO(10)\times\SU(3)\times\U(1)$ & $96$ & $60$ & $32$ & $6$\\
\hline
$\E_8/\SU(7)\times\SU(2)\times\U(1)$ & $84$ & $70$ & $28$ & $14$\\
\hline
\end{tabular}
\]}
\noindent {{\bf Case}  $\bold{r=5}$:} 
According to \cite{CS1}, there is only one flag manifold $M=G/K$ with $G$ simple, $b_{2}(M)=1$ and $r=5$; This is the coset space 
\[
M=G/K=\E_8/\U(1)\times \SU(4)\times \SU(5)
\]
and  is determined by painting black  the simple root   $\al_4$ of $\E_8$, i.e. $\Pi_{M}=\{\al_4\}$, with ${\rm Dynk}(\al_4)=5$ (and hence $r=5$). The Ricci components $\ric_{i}$ are given by  
{\small\[
\begin{array}{ll} 
\ric_1 &=  \displaystyle{\frac{1}{2x_1} -
\frac{c_{11}^{2}}{2\,d_1}  \frac{x_2}{{x_1}^2}} + \frac{c_{12}^{3}}{2\,d_1}
\biggl(\frac{x_1}{x_2 x_3} - \frac{x_2}{x_1 x_3}- \frac{x_3}{x_1 x_2}
\biggr)  
\displaystyle{ +  \frac{c_{13}^{4}}{2\,d_1}
\biggl(\frac{x_1}{x_3 x_4} - \frac{x_3}{x_1 x_4}- \frac{x_4}{x_1 x_3}
\biggr)}\\ &
\displaystyle{  +  \frac{c_{14}^{5}}{2\,d_1}
\biggl(\frac{x_1}{x_4 x_5} - \frac{x_4}{x_1 x_5}- \frac{x_5}{x_1 x_4}
\biggr)}\,,
\\ & \\
\ric_2 &=  \displaystyle{\frac{1}{2x_2} + 
\frac{c_{11}^{2}}{4\,d_2}}\biggl( \frac{x_2}{{x_1}^2} - \frac{2}{x_2}\biggr)
 - \frac{c_{22}^{4}}{2\,d_2}\frac{x_4}{{x_2}^2}  + \displaystyle{  \frac{c_{12}^{3}}{2\,d_2}
\biggl(\frac{x_2}{x_1 x_3} - \frac{x_1}{x_2 x_3}- \frac{x_3}{x_2 x_1}
\biggr)} %\\ & 
\displaystyle{ +  \frac{c_{23}^{5}}{2\,d_2}
\biggl(\frac{x_2}{x_3 x_5} - \frac{x_3}{x_2 x_5}- \frac{x_5}{x_2 x_3}
\biggr)}\,,
\\ & \\
\ric_3 &=  \displaystyle{\frac{1}{2 x_3} +}
\frac{c_{12}^{3}}{2\,d_3} 
\biggl(\frac{x_3}{x_1 x_2} - \frac{x_2}{x_3 x_1}- \frac{x_1}{x_3 x_2} \biggr) 
 +  \frac{c_{13}^{4}}{2\,d_3}
\biggl(\frac{x_3}{x_1 x_4} - \frac{x_1}{x_3 x_4}- \frac{x_4}{x_1 x_3}
\biggr) % \\ & 
 \displaystyle{+  \frac{c_{23}^{5}}{2\,d_3} 
\biggl(\frac{x_3}{x_2 x_5} - \frac{x_2}{x_3 x_5}- \frac{x_5}{x_3 x_2}
\biggr)}\,,
\\ & \\
\ric_4 & = \displaystyle{\frac{1}{2 x_4} + 
\frac{c_{22}^{4}}{4\,d_4}}\biggl( \frac{x_4}{{x_2}^2} - \frac{2}{x_4}\biggr)
 \displaystyle{  +  \frac{c_{13}^{4}}{2\,d_4}}
\biggl(\frac{x_4}{x_1 x_3} - \frac{x_1}{x_3 x_4}- \frac{x_3}{x_4 x_1}
\biggr)  
\displaystyle{ +  \frac{c_{14}^{5}}{2\,d_4} 
\biggl(\frac{x_4}{x_1 x_5} - \frac{x_1}{x_4 x_5}- \frac{x_5}{x_1 x_4}
\biggr)}\,, 
\\ & \\
\ric_5 & = \displaystyle{\frac{1}{2 x_5} +  
\frac{c_{23}^{5}}{2\,d_5}}
\left(\frac{x_5}{x_2 x_3} - \frac{x_2}{x_3 x_5}- \frac{x_3}{x_2 x_5}
\right)  +
 \displaystyle{\frac{c_{14}^{5}}{2\,d_5} 
\left(\frac{x_5}{x_1 x_4} - \frac{x_1}{x_4 x_5}- \frac{x_4}{x_1 x_5}
\right). } 
\end{array}
\]}
The non-zero  structure constants have been computed in \cite[Prop.~6]{CS1} and it is useful to recall them:
 \[
c_{11}^{2} = 12, \quad c_{12}^{3} = 8, \quad c_{13}^{4} = 4, \quad c_{14}^{5} =4/3, \quad c_{22}^{4} = 4, \quad c_{23}^{5}= 2.
 \]
 Moreover, $d_1  = 80$, $ d_2 = 60$, $ d_3  = 40$, $ d_4  = 20$ and $ d_5 = 8$. \\
\noindent {{\bf Case}  $\bold{r=6}$:}  
By \cite{CS1} it is known that there is also only one flag manifold $M=G/K$ with $G$ simple, $b_{2}(M)=1$ and $r=6$. This is isometric to the homogeneous space 
\[
M=G/K=\E_8/\U(1)\times \SU(2)\times \SU(3)\times \SU(5),
\]
 which   is determined by painting black  the simple root  $\al_5$ of $\E_8$, i.e. $\Pi_{M}=\{\al_5\}$, with ${\rm Dynk}(\al_5)=6$. We know that $d_1 = 60$, $d_2 = 60$, $ d_3 = 40$, $ d_4  = 30$, $ d_5 =12$ and $d_6=10$. Also, the values of the non-zero structure constants have the form (see \cite[Prop.~12]{CS1})
\[
 c_{11}^{2}=8, \quad  c_{12}^{3}=6, \quad c_{13}^{4}=4, \quad  c_{14}^{5}=2, \quad  c_{15}^{6}=1, \quad  c_{22}^{4}=6, \quad  c_{23}^{5}=2,  \quad  c_{24}^{6}=2, \quad  c_{33}^{6}=2,
\]
and the components $\ric_{i}$ of  the Ricci tensor $\Ric_{g}$ corresponding to $g$  are given by 
{\small
\[
\begin{array}{ll} 
\ric_1 &=  \displaystyle{\frac{1}{2x_1} -
\frac{c_{11}^{2}}{2\,d_1} \frac{x_2}{{x_1}^2}} + \frac{c_{12}^{3}}{2\,d_1}
\biggl(\frac{x_1}{x_2 x_3} - \frac{x_2}{x_1 x_3}- \frac{x_3}{x_1 x_2}
\biggr) +  \frac{c_{13}^{4}}{2\,d_1} %{1 \brack 34}
\biggl(\frac{x_1}{x_3 x_4} - \frac{x_3}{x_1 x_4}- \frac{x_4}{x_1 x_3}
\biggr) \\ & 
\displaystyle{  +  \frac{c_{14}^{5}}{2\,d_1}% {1 \brack 45}
\biggl(\frac{x_1}{x_4 x_5} - \frac{x_4}{x_1 x_5}- \frac{x_5}{x_1 x_4}
\biggr) + \frac{c_{15}^{6}}{2\,d_1} \biggl(\frac{x_1}{x_5 x_6} - \frac{x_5}{x_1 x_6}- \frac{x_6}{x_1 x_5}
\biggr)}\,,
\\ & \\
\ric_2 &=  \displaystyle{\frac{1}{2x_2} + 
\frac{c_{11}^{2}}{4\,d_2}\biggl( \frac{x_2}{{x_1}^2} - \frac{2}{x_2}\biggr)
 - \frac{c_{22}^{4}}{2\,d_2} \frac{x_4}{{x_2}^2}   +  \frac{c_{12}^{3}}{2\,d_2} % {2 \brack 13}
\biggl(\frac{x_2}{x_1 x_3} - \frac{x_1}{x_2 x_3}- \frac{x_3}{x_2 x_1}
\biggr)} \\ & 
\displaystyle{ +  \frac{c_{23}^{5}}{2\,d_2} %{2 \brack 35}
\biggl(\frac{x_2}{x_3 x_5} - \frac{x_3}{x_2 x_5}- \frac{x_5}{x_2 x_3}
\biggr)+  \frac{c_{24}^{6}}{2\,d_2} %{2 \brack 46}
\biggl(\frac{x_2}{x_4 x_6} - \frac{x_4}{x_2 x_6}- \frac{x_6}{x_2 x_4}
\biggr)}\,,
\\ & \\
\ric_3 &= \displaystyle{\frac{1}{2 x_3}} - \displaystyle{\frac{c_{33}^{6}}{2\,d_3}% {6 \brack 33}
\frac{x_6}{{x_3}^{2}}+}
\frac{c_{12}^{3}}{2\,d_3} %{3 \brack 12}
\biggl(\frac{x_3}{x_1 x_2} - \frac{x_2}{x_3 x_1}- \frac{x_1}{x_3 x_2} \biggr) 
 +  \frac{c_{13}^{4}}{2\,d_3}% {3\brack 14}
\biggl(\frac{x_3}{x_1 x_4} - \frac{x_1}{x_3 x_4}- \frac{x_4}{x_1 x_3}
\biggr) \\ & 
 \displaystyle{+  \frac{c_{23}^{5}}{2\,d_3}%{3 \brack 25}
\biggl(\frac{x_3}{x_2 x_5} - \frac{x_2}{x_3 x_5}- \frac{x_5}{x_3 x_2}
\biggr)}\,,
\\ & \\
\ric_4 & = \displaystyle{\frac{1}{2 x_4} + 
\frac{c_{22}^{4}}{4\,d_4}}\biggl( \frac{x_4}{{x_2}^2} - \frac{2}{x_4}\biggr)
 \displaystyle{  +  \frac{c_{13}^{4}}{2\,d_4} }% {4 \brack 13}}
\biggl(\frac{x_4}{x_1 x_3} - \frac{x_1}{x_3 x_4}- \frac{x_3}{x_4 x_1}
\biggr)  \\ & 
\displaystyle{ +  \frac{c_{14}^{5}}{2\,d_4} %{4 \brack 15}
\biggl(\frac{x_4}{x_1 x_5} - \frac{x_1}{x_4 x_5}- \frac{x_5}{x_1 x_4}
\biggr)+  \frac{c_{24}^{6}}{2\,d_4}% {4 \brack 26}
\biggl(\frac{x_4}{x_2 x_6} - \frac{x_2}{x_4 x_6}- \frac{x_6}{x_2 x_4}
\biggr)}\,,
\\ & \\
\ric_5 & = \displaystyle{\frac{1}{2 x_5} +  
\frac{c_{14}^{5}}{2\,d_5} }%{5 \brack 14}} 
\left(\frac{x_5}{x_1 x_4} - \frac{x_1}{x_4 x_5}- \frac{x_4}{x_1 x_5}
\right)+\displaystyle{\frac{c_{23}^{5}}{2\,d_5}}% {5 \brack23}}
\left(\frac{x_5}{x_2 x_3} - \frac{x_2}{x_3 x_5}- \frac{x_3}{x_2 x_5}
\right)  +
  \displaystyle{\frac{c_{15}^{6}}{2\,d_5}} %{5 \brack 16}}
\left(\frac{x_5}{x_1 x_6} - \frac{x_1}{x_5 x_6}- \frac{x_6}{x_1 x_5}
\right)\,, 
\\ & \\
\ric_6 & =\displaystyle{\frac{1}{2 x_6} + 
\frac{c_{33}^{6}}{4\,d_6} }%{6 \brack 33}}
\biggl( \frac{x_6}{{x_3}^2} - \frac{2}{x_6}\biggr)
 \displaystyle{  +  \frac{c_{15}^{6}}{2\,d_6} }%{6 \brack 15}}
\biggl(\frac{x_6}{x_1 x_5} - \frac{x_1}{x_5 x_6}- \frac{x_5}{x_1 x_6}
\biggr)  %\\ & 
\displaystyle{ +  \frac{c_{24}^{6}}{2\,d_6} %{6 \brack 24}
\biggl(\frac{x_6}{x_2 x_4} - \frac{x_2}{x_4 x_6}- \frac{x_4}{x_2 x_6}
\biggr)} .
\end{array}
\]}
\subsection{The main theorem}
Let $M=G/K$ be a flag manifold with $b_{2}(M)=1$ and $2\leq r\leq 6$. The system of the homogeneous Ricci flow is given by
\begin{equation}\label{hrf}
\Big\{\dot{x}_{i}=-2x_{i}\cdot\ric_{i} \ : \  i=1, \ldots, r\Big\}.
\end{equation}
For any case separately, a direct computation  shows that system (\ref{hrf}) does {\it not} possess   fixed points in $\M^G\cong \R^r_+$,  i.e. points $(x_1, \ldots, x_r)\in\R^{r}_{r}$ satisfying the system $\{\dot{x}_i=\dot{x}_2=\cdots=\dot{x}_{r}=0\}$, which verifies   Lemma \ref{nofixed}.

Let us now agree on the following notation: 
We denote by $\fr{e}_j$   a fixed point of the homogeneous Ricci flow (HRF)  at {\it infinity of $\M^G$}, as defined in Remark \ref{fixinfinity}.  We shall write $N$ for the number of all such fixed points.   We will also  denote  by $d_{j}^{\rm unstb}$ (respectively $d_{j}^{\rm stb}$) the dimension of the {\it unstable manifold} (respectively {\it stable manifold}) in $\M^{G}$ (respectively in the infinity of $\M^G$), corresponding to $\fr{e}_j$. In this terms we obtain the following

\begin{theorem}\label{mainthem}  
Let $M=G/K$ be a non-symmetric  flag manifold with $b_{2}(M)=1$, and let $r$ ($2\leq r\leq 6$) be the  number of the corresponding isotropy summands. Then, the following hold:\\
(1)The HRF admits exactly $N$ fixed points $\fr{e}_j$ at the infinity of $\M^G$, where for any coset $G/K$ the number $N$ is specified in Table \ref{Table1}.  These fixed points are in bijective correspondence with non-isometric invariant Einstein metrics on $M=G/K$, and are specified explicitly in the proof. \\
(2)The dimensions of the stable/unstable manifolds corresponding to  $\fr{e}_j$ are given  in Table \ref{Table1}, where $\fr{e}_1$   represents  the fixed point corresponding to the  unique invariant K\"ahler-Einstein metric on $G/K$. We see that\\
i) For any $M=G/K$, the fixed point $\fr{e}_1$ has always an 1-dimensional unstable manifold in $\M^G$, and a $(r-1)$-dimensional stable manifold in the infinity of $\M^G$.\\
ii) Any other fixed point  $\fr{e}_k$ with $2\leq k\leq N$  has always a 2-dimensional unstable manifold, while its stable manifold is $(r-2)$-dimensional and is contained in the infinity of $\M^G$, with the following three exceptions: 
%has always  a 2-dimensional unstable manifold in $\M^G$, and a $(r-2)$-dimensional stable manifold in the infinity of $\M^G$, with the following three exceptions: 
\begin{itemize}
\item the fixed point   $\fr{e}_5$ for the space $M_{*}:=\E_8/\U(1)\times\SU(3)\times\SO(10)$ in the case with $r=4$;
\item the fixed point $\fr{e}_5$ in the case of $M=\E_8/\U(1)\times \SU(4)\times \SU(5)$ with $r=5$; 
\item the fixed point  $\fr{e}_4$ in the case of  $M=\E_8/\U(1)\times \SU(2)\times \SU(3)\times \SU(5)$ with $r=6$.
\end{itemize}
These three exceptions have a 3-dimensional unstable manifold and a $(r-3)$-dimensional stable manifold. Similarly, the stable manifold for these cases is   contained entirely in the infinity of $\M^G$.
\\
(3) Each unstable manifold  of any $\fr{e}_j$,   contains a non-collapsed ancient solution, given by
\[
g_{j} : (-\infty, \frac{1}{2\lambda_j})\longrightarrow\M^G\,,\quad t\longmapsto g_{j}(t)=(1-2\lambda_{j}t)\cdot\fr{e}_{j}\,,\quad j=1,\ldots, N\,,
\]
 where $\lambda_j$ is the  Einstein constant  of the  corresponding Einstein metric $g_{j}(0)$, $j=1,\ldots N$, on $M=G/K$ (these are also specified below).
All such solutions $g_{j}(t)$ tend to $0$ when $t\to T=\frac{1}{2\lambda_j}>0$, and   $M=G/K$ shrinks to a point in finite time. \\  %Note that in Table \ref{Table1},  $M_{*}$ denotes the homogeneous space $M_{*}:=\E_8/\SO(10)\times\SU(3)\times\U(1)$, with $r=4$.\\
(4) When $r=2$,  any other possible solution  of the Ricci flow with initial condition in $\M^G$ has $\{0\}$ as its  $\omega$--limit set.
\begin{table}[ht]
\centering
\renewcommand\arraystretch{1.4}
\begin{tabular} {r | c | c | c  | c | c  | c }
$r$ & conditions for $M=G/K$ & $N$ & $d_{1}^{\rm unstb}$ & $d_{1}^{\rm stb}$ & $d_{j}^{\rm unstb}$ $(2\leq j\leq N)$ & $d_{j}^{\rm stb}$ $(2\leq  j\leq N)$\\
\hline
2 & - & 2 & 1 & 1 & 2 & 0 \\
3 & - & 3 & 1 & 2 & 2 & 1 \\
4  &$M\ncong M_{*}$&   3   & 1  & 3 & 2 &  2 \\
4 & $M\cong M_{*}$&  5      &  1 &  3 & $(\text{for} \ j\neq 5)$ \  2 & $( \text{for} \ j\neq 5)$ \ 2 \\ 
   &                            &        &     &     &  $(\text{for} \ j=5)$ \ 3 &       $(\text{for} \ j=5)$ \ 1 \\
5 &-    & 6 &   1 & 4 & $(\text{for} \ j\neq 5)$ \ 2  & $(\text{for} \ j\neq 5)$ \ 3 \\
   &      &    &      &   &   $(\text{for} \ j=5)$ \ 3   & $(\text{for} \ j=5)$ \ 2 \\
6 &-    & 5&    1 & 5 & $(\text{for} \ j\neq 4)$ 2  & $(\text{for} \ j\neq 4)$ 4\\
   &      &  &       &    &  $(\text{for} \ j= 4)$ 3  & $(\text{for} \ j=4)$ 3\\
\thickline    
\end{tabular}
\caption{The exact number $N$ of the fixed points $\fr{e}_{k}$ of HRF at infinity of $\M^G$ for any non-symmetric flag space $M=G/K$ with $b_2(M)=1$, and the dimensions $d_{k}^{\rm stb}, d_{k}^{\rm unstb}$, for any  $1\leq k\leq N$.}\label{Table1}
\end{table}
\end{theorem}
\begin{proof}
We split the proof in cases, depending on the possible values of $r$. \\
\noindent  {\bf Case} $\bold{r=2.}$  In this case the system (\ref{hrf}) reduces to:
 \begin{equation}\label{hf2}
\Big\{\dot{x}_1=-\frac{(d_1+4d_2)x_1-d_2x_2}{(d_1+4d_2)x_1}\,,\quad
\dot{x}_2=-\frac{8d_2x_1^2+d_1x_2^2}{2(d_1+4d_2)x_1^2} \Big\}.
\end{equation}
To search for fixed points of the HRF at infinity of $\M^G$, we first multiply the   right-hand side of these equations with the positive factor $2(d_1+4d_2)x_1^2$. This multiplication does not qualitatively affect system's behaviour, and  we result with the following equivalent system:
\[
\Big\{
\dot{x}_1=\RF_1(x_1,x_2)=-2(d_1+4d_2)x_1^2+2d_2x_1x_2\,,\quad
\dot{x}_2=\RF_2(x_1,x_2)=-(8d_2x_1^2+d_1x_2^2)
\Big\},
\]
where the right-hand side consists of two homogeneous polynomials of degree $2$.  This is the maximal degree $d$ of the system, as discussed above.  We can therefore apply the Poincar\'{e} compactification procedure to study its behaviour at infinity.  For the formulas given in Example \ref{inf2d},  i.e.
\[
\Big\{\dot{x}_1=x_2^2\big(-x_1\overline{\RF}_1(x_1, x_2)+\overline{\RF}_{2}(x_1, x_2)\big)\,,\quad \dot{x}_2=x_2^2\big(-x_2\overline{\RF}_1(x_1, x_2)\big)\Big\},
\]
we compute
\[
\overline{\RF}_1(x_1,x_2)=\RF_{1}(\frac{1}{x_2},\frac{x_1}{x_2})=-\frac{2(d_1+4d_2-d_2x_1)}{x_2^{2}}\,,\quad \overline{\RF}_2(x_1,x_2)=\RF_{2}(\frac{1}{x_2},\frac{x_1}{x_2})=-\frac{(8 d_2+d_1x_1^2)}{x_2^2}\,.
\]
Hence finally we result with the system
\[
\Big\{\dot{x}_1=-(x_1-2)(-4d_2+d_1x_1+2d_2x_1)\,,\quad
\dot{x}_2=2(d_1+4d_2-d_2x_1)x_2\Big\},
\]
which is the desired expression   in the $U$ chart. To study  the behaviour  of this system at infinity of $\R^{2}_{+}$, we set $x_2=0$. Then, the second equation becomes $\dot{x}_2=0$, confirming that infinity remains invariant under the flow. To locate fixed points, we solve the equation 
\[
\dot{x}_1= h(x_1)=-(x_1-2)(-4d_2+d_1x_1+2d_2x_1)=0\,.
\]
We obtain two exactly solutions, namely $x_1^{a}=2$ and $x_1^{b}=4d_2/(d_1+2d_2)$, and since $h'(x_1^{a})=-2d_1$, $h'(x_1^{b})=2d_1$, both  fixed points are hyperbolic. 
In particular, $h'(x_1^a)<0$ and so  $x_1^{a}$ is always an attracting node with eigenvalue equal to $-2d_1$. On the other hand,  $h'(x_1^b)>0$ and hence  $x_1^{b}$ is a repelling node, with eigenvalue equal to $2d_1$. 

Recall now that the central projection $F$ maps the sphere to the $y_1=1$ plane. Therefore, the fixed points $x_1^a,\ x_1^b$ represent the points $\fr{e}_1=(1,x_1^a)$ and $\fr{e}_2=(1,x_1^b)$ in $\R^{2}_{+}$,  which correspond to the invariant Einstein metrics
$$g_{KE}=1\cdot B|_{\mathfrak{m}_1}+2\cdot B|_{\mathfrak{m}_2},\quad g_{E}=1\cdot B|_{\mathfrak{m}_1}+4d_2/(d_1+2d_2)\cdot B|_{\mathfrak{m}_2},$$
respectively (see also \cite{Chry211}).
To locate the ancient solutions, let $\hat{g}(t)=t(x_1(0),x_2(0))$ be a straight line in $\M^{G}$.  At the point $(a,b)$, belonging in the trace of $\hat{g}(t)$, the vector normal to the straight line is the vector $(-b,a)$.   The straight line $\hat{g}(t)$ must be tangent to the vector field $X(x_1,x_2)$ defined by the homogeneous Ricci flow, hence the following equation should hold:
\[
(-b,a)\cdot X(a,b)=0\,.
\]
This gives us two solutions, namely $(a_1,b_1)=(1,2)$ and $(a_2,b_2)=(1,4d_2/(d_1+2d_2))$, which define the lines $\gamma_i(t)=t(a_i,b_i),\ i=1,2$. The solutions $g_{1}(t)$, $g_{2}(t)$ of system (\ref{hf2}) corresponding to the lines $\gamma_{1}(t)$ and $\gamma_{2}(t)$, are determined by the  following equations
\begin{eqnarray*}
g_{1}(t)&=&(1-2\lambda_{1}t)\cdot (1, 2)=(1-2\lambda_{1}t)\cdot\fr{e}_1\,, \\
g_{2}(t)&=&(1-2\lambda_{2}t)\cdot (1, \frac{4d_2}{d_1+2d_2})=(1-2\lambda_{2}t)\cdot \fr{e}_2\,,
\end{eqnarray*}
where  $\lambda_1, \lambda_2$ are the Einstein constants of  $g_{KE}=g_{1}(0)$ and   $g_{E}=g_{2}(0)$, respectively, given by
 \[
 \lambda_1=\frac{d_1 + 2 d_2}{2(d_1 + 4d_2)}\,,\quad\quad\lambda_2=\frac{d_1^2 + 6 d_1 d_2 + 4 d_2^2}{2(d_1 + 2 d_2) (d_1 + 4 d_2)}.
 \]
 Obviously, these are both ancient solutions since are defined on the open set $(-\infty, \frac{1}{2\lambda_i})$ (see also  Proposition \ref{ak}),  in particular $g_{i}(t)\to 0$ when  $t\to\frac{1}{2\lambda_i},\ i=1,2$. The assertion that $g_{i}(t)$ are non-collapsed follows by Corollary  \ref{noncollapsed}. 
 
Based on the definitions of the central projections $f$ and $F$  given before, we verify that
\[
F(f(\gamma_1(t)))=(1,2)\,,\quad\text{and} \quad F(f(\gamma_2(t)))=(1,\frac{4d_2}{d_1+2d_2}).
\]
Thus, we have that
\[
\lim_{t\rightarrow \frac{1}{2\lambda_i}}\gamma_i(t)=0\,,\quad\text{and}\quad\lim_{t\rightarrow -\infty}\gamma_i(t)=\fr{e}_i\,,\quad \forall \ i=1,2,
\]
which proves the claim that these ancient solitons belong to the unstable manifolds of the fixed points located at infinity.

To verify the claim in (4), we use the function $V_2(x_1,x_2)=x_1^2+x_2^2$ as a Lyapunov function for the system (\ref{hf2}). We compute that
\[
\frac{dV}{dt}(x_1(t),x_2(t))=-\frac{2 d_2 x_1^2 (4 x_1 + 3 x_2) + d_1 (2 x_1^3 + x_2^3)}{2 (d_1 + 4 d_2) x_1^2},
\]
which for $x_1,x_2>0$ is always negative. Hence, if we denote as $(t_a,t_b),\ t_a,t_b\in \R\cup \{\pm \infty\}$ the domain of definition of the solution curve $g(t)=(x_1(t),x_2(t))$, we conclude that $g(t)$ tends to the origin, as $t\rightarrow t_b$. This completes the proof for $r=2$.\\
\noindent  {\bf Case} $\bold{r=3.}$  In this case,  to reduce   system (\ref{hrf})  in a polynomial dynamical system, we must multiply with the positive factor $2d_1d_2(d_1+4d_2+9d_3)x^2yz$. This gives 
\begin{eqnarray}
\label{sysrf3comp}
\dot{x}_1&=&\RF_1(x_1,x_2,x_3)=-2d_2x_1(d_1d_3x_1^3+d_2d_3x_1^3-d_1d_3x_1x_2^2-d_2d_3x_1x_2^2+ \nonumber\\
&+&d_1^2x_1x_2x_3+4d_1d_2x_1x_2x_3+9d_1d_3x_1x_2x_3-d_1d_2x_2^2x_3-2d_1d_3x_2^2x_3+ \nonumber\\
&+&d_2d_3x_2^2x_3-d_1d_3x_1x_3^2-d_2d_3x_1x_3^2)\,,\nonumber \\
\dot{x}_2&=&\RF_2(x_1,x_2,x_3)=-d_1x_2(-2d_1d_3x_1^3-2d_2d_3x_1^3+2d_1d_3x_1x_2^2+2d_2d_3x_1x_2^2+\\
&+& 8d_2^2x_1^2x_3-4d_1d_3x_1^2x_3+20d_2d_3x_1^2x_3+d_1d_2x_2^2x_3+2d_1d_3x_2^2x_3-d_2d_3x_2^2x_3- \nonumber \\
&-&2d_1d_3x_1x_3^2-2d_2d_3x_1x_3^2)\,, \nonumber \\
\dot{x}_3&=&\RF_3(x_1,x_2,x_3)=2d_1d_2x_1x_3(d_1x_1^2+d_2x_1^2-d_1x_1x_2-4d_2x_1x_2- \nonumber \\
&-&9d_3x_1x_2+d_1x_2^2+d_2x_2^2-d_1x_3^2-d_2x_3^2)\,. \nonumber
\end{eqnarray}
The right-hand side of the system above consists of homogeneous polynomials of degree $4$. Let us apply the Poincar\'{e} compactification procedure and set $x_3=0$, to obtain the equations governing the behaviour of the system at infinity of $\M^G$. By Example \ref{inf2d} we deduce that in the $U$-chart, these must be given as follows:
\begin{eqnarray*}
\dot{x}_1&=&x_1(2d_1^2d_3+4d_1d_2d_3+2d_2^2d_3-2d_1^2d_3x_1^2-4d_1d_2d_3x_1^2-2d_2^2 d_3x_1^2-8d_1d_2^2x_2+4d_1^2d_3x_2-\\
&-&20d_1d_2d_3x_2+2d_1^2d_2x_1x_2+8d_1d_2^2 x_1x_2+18d_1d_2d_3x_1x_2-d_1^2d_2x_1^2x_2-2d_1d_2^2x_1^2x_2-\\
&-&2d_1^2d_3x_1^2x_2-3d_1d_2d_3x_1^2x_2+2d_2^2d_3x_1^2x_2+
2d_1^2d_3x_2^2-2d_2^2d_3x_2^2)\,,\\
\dot{x}_2&=&-2d_2x_2(-d_1^2-d_1d_2-d_1d_3-d_2d_3+d_1^2x_1+4d_1d_2x_1+9d_1d_3 x_1-d_1^2x_1^2-d_1d_2x_1^2+\\
&+&d_1d_3x_1^2+d_2d_3x_1^2-d_1^2x_1x_2-4d_1d_2x_1x_2 -9d_1d_3x_1x_2+d_1d_2x_1^2x_2+2d_1d_3x_1^2x_2-\\
&-&d_2d_3x_1^2x_2+d_1^2x_2^2+d_1 d_2x_2^2+d_1d_3x_2^2+d_2d_3x_2^2)\,,\\
\dot{x}_3&=&0\,.
\end{eqnarray*}
As before, the last equation confirms that infinity remains invariant under the flow of the system.
To locate fixed points, we have to solve the system of equations $\{\dot{x}_1=\dot{x}_2=0\}$. 
For this, it is sufficient to   study each case separately and replace the dimensions $d_i$. For any case we  get exactly three fixed points, which we list as follows:
\begin{itemize}
\item {$\E_8/\E_6\times\SU(2)\times\U(1):$ In this case, we have $d_1=108, d_2=54, d_3=4$ and the fixed points at infinity are located at:
$$(2,3),\ (0.914286,1.54198),\ (1.0049,0.129681).$$}
\item {$\E_8/\SU(8)\times\U(1):$ In this case, we have $d_1=112, d_2=56 ,d_3=16$ and the fixed points at infinity are located at:
$$(2,3),\ (0.717586,1.25432),\ (1.06853,0.473177).$$}
\item {$\E_7/\SU(5)\times\SU(3) \times\U(1):$ In this case, we have $d_1=60, d_2=30, d_3=8$ and the fixed points at infinity are located at:
$$(2,3),\ (0.733552,1.27681),\ (1.06029,0.443559).$$}
\item {$\E_7/\SU(6)\times\SU(2)\times\U(1):$ In this case, we have $d_1=60, d_2=30, d_3=4$ and the fixed points at infinity are located at:
$$(2,3),\ (0.85368,1.45259),\ (1.01573,0.229231).$$}
\item {$\E_6/\SU(3)\times\SU(3)\times\SU(2)\times\U(1):$ In this case, we have $d_1=36, d_2=18, d_3=4$ and the fixed points at infinity are located at:
$$(2,3),\ (0.771752,1.33186),\ (1.04268,0.373467).$$}
\item {$\F_4/\SU(3)\times\SU(2)\times\U(1):$ In this case, we have $d_1=24, d_2=12, d_3=4$ and the fixed points at infinity are located at:
$$(2,3),\ (0.678535,1.20122),\ (1.09057,0.546045).$$}
\item {$\G_2/\U(2):$ In this case, we have $d_1=4, d_2=2, d_3=4$ and the fixed points at infinity are located at:
$$(2,3),\ (1.67467,2.05238),\ (0.186894,0.981478).$$}
\end{itemize}
The projections of these solutions via $F:U\rightarrow \mathbb{R}^3_{\geq}$, give us the points $\fr{e}_i$, $i=1, 2, 3$. These are the points obtained from the coordinates of the fixed points given above, with one extra coordinate equal to $1$, in the first entry.

According to (\ref{KE}), the indicated fixed point $\fr{e}_1=(1, 2,3)$  corresponds to the unique invariant K\"ahler-Einstein metric, and simple eigenvalue calculations show that it possesses a 2-dimensional stable manifold at the infinity of $\M^G$, and a 1-dimensional unstable manifold, which is contained in $\M^G$.  All  the other fixed points, correspond to non-K\"{a}hler non-isometric invariant Einstein metrics (see \cite{Kim90, AC11}), and have a 2-dimensional unstable manifold  and a 1-dimensional stable manifold, contained at the infinity of $\M^G$. \\
Let us now consider an invariant line of system (\ref{sysrf3comp}), of the form $\hat{g}(t)=t(x_1(0),x_2(0),x_3(0))$. At a point $(a,b,c)$ belonging to this line, the normal vectors are: $(-b,a,0),(0,-c,b)$, thus $\hat{g}(t)$ is tangent to the vector field $X(x_1,x_2,x_3)$ defined by the homogeneous Ricci flow if 
\[
(-b,a,0)\cdot X(a,b,c)=0\,,\quad (0,-c,b)\cdot X(a,b,c)=0\,.
\]
These equations possess three solutions, with respect to $(a,b,c)$; One of them is always $(1,2,3)$, while the other two can be obtained after numerically solving equations above for every value of $d_1,d_2,d_3$. 
These solutions give us the three non-collapsed ancient solutions $g_{i}(t)=(1-2\lambda_{i}t)\cdot\fr{e}_i$ with $t\in(-\infty, \frac{1}{2\lambda_i})$, for any $i=1, 2, 3$, where the Einstein constants $\lambda_i$ for the cases $i=2, 3$ can be computed easily by replacing the corresponding Einstein metrics $g_{2}(0), g_{3}(0)$   in the Ricci components, while for $g_{1}(t)$, $\lambda_1$ is specified as follows:
\[
\lambda_1=\frac{d_1 + 2 d_2 + 3 d_3}{2 d_1 + 8 d_2 + 18 d_3}\,.
\]
Moreover,   it is easy to show, taking limits, that the $\alpha$-limit set of each of these solutions is one of the located fixed points at infinity of $\M^G$, while the $\omega$-limit set, for all of them, is $\{0\}$.  Finally, for any $i$ we have $g_{i}(t)\to 0$ as $t\to\frac{1}{2\lambda_i}$, where $T=\frac{1}{2\lambda_i}$  depends on  the dimensions $d_i$, $i=1,2 ,3$ and so on the flag manifold $M=G/K$, while again  the assertion that  $g_{i}(t)$ are non-collapsed, for any $i=1, 2, 3,$ is a consequence of Corollary \ref{noncollapsed}.\\
\noindent  {\bf Case} $\bold{r=4.}$    The proof follows the lines of the previous cases $r=2, 3$. Thus,  we avoid to present similar arguments. Consider for example the case of the flag manifold $M=G/K$ with $b_2(M)=1$ and $r=4$, corresponding to $M_{*}:=\E_8/\U(1)\times\SU(3)\times\SO(10)$.  After multiplication with the positive   term $60x_1^2x_2^2x_3x_4$, the system (\ref{hrf}) of HRF  equation turns into the following polynomial system:
\begin{eqnarray*}
\dot{x}_1&=&x_1 x_2(-5 x_4 x_1^3 + x_4^2 x_1 x_2 - x_1^3 x_2 + 5 x_4 x_1 x_2^2 - 60 x_4 x_1 x_2 x_3 + 
   10 x_4 x_2^2 x_3 + 5 x_4 x_1 x_3^2 + x_1 x_2 x_3^2)\,,\\
\dot{x}_2&=&2 x_2 x_4 (4 x_1^3 x_2 - 4 x_1 x_2^3 + x_4 x_1^2 x_3 - 22 x_1^2 x_2 x_3 - 4 x_2^3 x_3 + 4 x_1 x_2 x_3^2)\,,\\
\dot{x}_3&=&3 x_1 x_2 x_3 (5 x_4 x_1^2 + x_4^2 x_2 - 20 x_4 x_1 x_2 + x_1^2 x_2 + 5 x_4 x_2^2 - 5 x_4 x_3^2 - 
   x_2 x_3^2)\,,\\
\dot{x}_4&=&-2 x_1 x_4 (8 x_4^2 x_2^2 - 8 x_1^2 x_2^2 + 5 x_4^2 x_1 x_3 + 20 x_1 x_2^2 x_3 - 8 x_2^2 x_3^2)\,.
\end{eqnarray*}
Note that the right-hand side of the equations above are  all homogeneous polynomials of degree $6$.
Hence we can apply  Proposition \ref{lochf}  to study the behaviour of this system at infinity  of $\M^{G}$. In the  $U$ chart and  by setting $x_4=0$, we obtain 
\begin{eqnarray*}
\dot{x}_1&=&-x_1 (-x_1^2 + x_1^2 x_2^2 - 13 x_1 x_3 + 13 x_1^3 x_3 + 44 x_1 x_2 x_3 - 60 x_1^2 x_2 x_3 + 
   18 x_1^3 x_2 x_3 - 3 x_1 x_2^2 x_3\\ &\ & + x_1^2 x_3^2 - 2 x_2 x_3^2)\,,\\
\dot{x}_2&=&-2 x_1 x_2 (-2 x_1 + 2 x_1 x_2^2 - 10 x_3 + 30 x_1 x_3 - 5 x_1^2 x_3 - 30 x_1 x_2 x_3 + 
   5 x_1^2 x_2 x_3 + 10 x_2^2 x_3\\ &\ & - x_1 x_3^2)\,,\\
\dot{x}_3&=&-x_3 (-17 x_1^2 + 40 x_1^2 x_2 - 15 x_1^2 x_2^2 - 5 x_1 x_3 + 5 x_1^3 x_3 - 60 x_1^2 x_2 x_3 + 
   10 x_1^3 x_2 x_3 + 5 x_1 x_2^2 x_3\\ &\ & + 17 x_1^2 x_3^2 + 10 x_2 x_3^2)\,,\\
\dot{x}_4&=&0\,.
\end{eqnarray*}
Again, last equation confirms that infinity remains invariant under the flow of the system.
Solving the system $\{\dot{x}_1=\dot{x}_2=\dot{x}_3=0\}$, we get exactly five fixed points, given by
$$(2,3,4)\,, \ (1.09705,0.770347,1.29696)\,,  \ (1.15607,1.01783,0.214618)\,,$$
$$
(0.649612,1.10943,1.06103)\,, \ (0.763357,1.00902,0.191009)\,.$$
These solutions give us, through the central projection $F$ of the $U$ chart on $\mathbb{R}^4_{\geq}$, the five fixed points $\fr{e}_1, \fr{e}_2, \fr{e}_3,\ \fr{e}_4,\ \fr{e}_5$. The fixed point $\fr{e}_1=(1,2,3,4)$ corresponds to the unique invariant K\"{a}hler-Einstein metric. Consider the Jacobian matrix corresponding to the system in the $U$ chart, that is
\[
\mathsf{Jac}:=\begin{pmatrix}
\frac{\partial f_1}{\partial x_1}  & \frac{\partial f_1}{\partial x_2} & \frac{\partial f_1}{\partial x_3} \\
\frac{\partial f_2}{\partial x_1} & \frac{\partial f_2}{\partial x_2} & \frac{\partial f_2}{\partial x_3} \\
\frac{\partial f_3}{\partial x_1} & \frac{\partial f_3}{\partial x_2} & \frac{\partial f_3}{\partial x_3}
\end{pmatrix}
\]
where
\begin{eqnarray*}
f_1&:=&-x_1 (-x_1^2 + x_1^2 x_2^2 - 13 x_1 x_3 + 13 x_1^3 x_3 + 44 x_1 x_2 x_3 - 60 x_1^2 x_2 x_3 + 
   18 x_1^3 x_2 x_3 - 3 x_1 x_2^2 x_3\\ &\ & + x_1^2 x_3^2 - 2 x_2 x_3^2)\,,\\
f_2&:=&   -2 x_1 x_2 (-2 x_1 + 2 x_1 x_2^2 - 10 x_3 + 30 x_1 x_3 - 5 x_1^2 x_3 - 30 x_1 x_2 x_3 + 
   5 x_1^2 x_2 x_3 + 10 x_2^2 x_3\\ &\ & - x_1 x_3^2)\,,\\
f_3&:=&-x_3 (-17 x_1^2 + 40 x_1^2 x_2 - 15 x_1^2 x_2^2 - 5 x_1 x_3 + 5 x_1^3 x_3 - 60 x_1^2 x_2 x_3 + 
   10 x_1^3 x_2 x_3 + 5 x_1 x_2^2 x_3\\ &\ & + 17 x_1^2 x_3^2 + 10 x_2 x_3^2)\,.
\end{eqnarray*}
Calculating the Jacobian matrix %of the right-hand side of the previous system 
at the corresponding fixed point $(2,3,4)$, we find it to be equal to:
\[
\mathsf{Jac}_{(2, 3, 4)}=\begin{pmatrix}
-1600 &  368 &  32\\
960 & -1248 & 192\\
320 & 1760 & -1696
\end{pmatrix}.
\]
This has 3 negative eigenvalues, thus we conclude that the fixed point $\fr{e}_1$ possesses a 3-dimensional stable manifold, contained in the infinity of $\M^G$, while the straight line joining $\fr{e}_1$ with the origin of $\M^G$ corresponds to its 1-dimensional unstable manifold.

Computing the Jacobian matrix at the other fixed points and calculating their eigenvalues, we conclude that all the other fixed points, corresponding to non-K\"ahler, non-isometric invariant Einstein metrics (see \cite{Chry210}), have a 2-dimensional stable manifold, located in the infinity of $\M^G$, and a 2-dimensional unstable manifold, one direction of which is the straight line in $\M^G$ tending towards the origin, with the exception of fixed point $\fr{e}_5$. The Jacobian matrix there becomes
\[
\mathsf{Jac}_{(0.763357,1.00902,0.191009)}=\begin{pmatrix}
1.33763 & -1.13506 & 0.251494 \\
0.0597364 & -4.80309 & 0.447799\\
-0.13328 & -0.320307 & 3.98402
\end{pmatrix}
\]
which has one negative eigenvalue and two positive ones. Thus, $\fr{e}_5$ possesses a 1-dimensional stable manifold, contained in the infinity of $\M^G$, and a 3-dimensional unstable manifold, one direction of which is the straight line emanating from $\fr{e}_5$ and tending to the origin of $\M^G$.  Invariant lines, corresponding to non-collapsed ancient solutions $g_{i}(t)=(1-2\lambda_i t)\cdot\fr{e}_i$, $t\in(-\infty, \frac{1}{2\lambda_i})$, can be found as in the previous cases, confirming once again that their $\omega$--limit sets are equal to $\{0\}$, while the $\alpha$--limit set is the corresponding fixed point at infinity of $\M^G$. This  proves our claims. 
 The Ricci flow equations on the rest three homogeneous spaces of that type can be treated similarly. \\
\noindent  {\bf Case} $\bold{r=5.}$  After multiplication with the positive   term $60x_{1}^2 x_{2}^2 x_{3}x_{4}x_{5}$, system (\ref{hrf}) turns into the following polynomial system:
\begin{eqnarray*}
\dot{x}_1&=&-x_1x_2(x_1^3 x_2 x_3 - x_1 x_2 x_3 x_4^2 + 3 x_1^3 x_2 x_5 - 
   3 x_1 x_2 x_3^2 x_5 + 6 x_1^3 x_4 x_5 - 6 x_1 x_2^2 x_4 x_5 +\\ &\ &+
   60 x_1 x_2 x_3 x_4 x_5 - 9 x2^2 x_3 x_4 x_5 - 6 x_1 x_3^2 x_4 x_5 - 
   3 x_1 x_2 x_4^2 x_5 - x_1 x_2 x_3 x_5^2)\,,\\
\dot{x}_2&=&2x_2x_4(-x_1^2 x_2^3 + x_1^2 x_2 x_3^2 + 4 x_1^3 x_2 x_5 - 4 x_1 x_2^3 x_5 - 
   24 x_1^2 x_2 x_3 x_5 - 3 x_2^3 x_3 x_5 + 4 x_1 x_2 x_3^2 x_5 +\\ &\ &+ 
   2 x_1^2 x_3 x_4 x_5 + x_1^2 x_2 x_5^2)\,,\\
\dot{x}_3&=&3x_1x_2x_3(x_1 x_2^2 x_4 - x_1 x_3^2 x_4 + 2 x_1^2 x_2 x_5 - 2 x_2 x_3^2 x_5 + 
   4 x_1^2 x_4 x_5 - 20 x_1 x_2 x_4 x_5 + 4 x_2^2 x_4 x_5 - \\ &\ & -4 x_3^2 x_4 x_5 + 
   2 x_2 x_4^2 x_5 + x_1 x_4 x_5^2)\,,\\
\dot{x}_4&=&-2x_1x_4(-2 x_1^2 x_2^2 x_3 + 2 x_2^2 x_3 x_4^2 - 6 x_1^2 x_2^2 x_5 + 
   24 x_1 x_2^2 x_3 x_5 - 6 x_2^2 x_3^2 x_5 +6 x_2^2 x_4^2 x_5 + \\ &\ &+ 
   3 x_1 x_3 x_4^2 x_5 - 2 x_2^2 x_3 x_5^2)\,,\\
\dot{x}_5&=&5x_1x_2x_5(2 x_1^2 x_2 x_3 + 3 x_1 x_2^2 x_4 - 12 x_1 x_2 x_3 x_4 + 
   3 x_1 x_3^2 x_4 + 2 x_2 x_3 x_4^2 - 2 x_2 x_3 x_5^2 - 3 x_1 x_4 x_5^2)\,.   
\end{eqnarray*}
Note that the right-hand side consists of homogeneous polynomials of degree $7$ and we can use  Proposition \ref{lochf}  to study the behaviour of this system at infinity  of $\M^{G}$.  Using the expressions given above, the system at infinity, written in the $U$ chart and setting $x_5=0$, reads as follows:
\begin{eqnarray*}
\dot{x}_1&=&-x_1 (-x_1^2 x_2 + 2 x_1^3 x_3 - 2 x_1 x_2^2 x_3 + x_1^2 x_2 x_3^2 - 3 x_1^2 x_4 + 
   3 x_1^2 x_2^2 x_4 - 14 x_1 x_3 x_4 + 14 x_1^3 x_3 x_4+ \\ &\ & + 48 x_1 x_2 x_3 x_4 - 
   60 x_1^2 x_2 x_3 x_4 + 15 x_1^3 x_2 x_3 x_4 - 2 x_1 x_2^2 x_3 x_4 + 
   3 x_1^2 x_3^2 x_4 - 4 x_2 x_3^2 x_4+ \\ &\ & + x_1^2 x_2 x_4^2 - 2 x_1 x_3 x_4^2)\,,\\
\dot{x}_2&=&-x_1 x_2 (-x_1 x_2 - 3 x_1^2 x_3 + 3 x_2^2 x_3 + x_1 x_2 x_3^2 - 9 x_1 x_4 + 
   9 x_1 x_2^2 x_4 - 18 x_3 x_4 + 60 x_1 x_3 x_4- \\ &\ & - 6 x_1^2 x_3 x_4 - 
   60 x_1 x_2 x_3 x_4 + 9 x_1^2 x_2 x_3 x_4 + 18 x_2^2 x_3 x_4 - 3 x_1 x_3^2 x_4 + 
   x_1 x_2 x_4^2 - 3 x_3 x_4^2)\,,\\
\dot{x_3}&=&-x_3 (-5 x_1^2 x_2 + 5 x_1^2 x_2 x_3^2 - 15 x_1^2 x_4 + 48 x_1^2 x_2 x_4 - 
   9 x_1^2 x_2^2 x_4 - 6 x_1 x_3 x_4 + 6 x_1^3 x_3 x_4- \\ &\ & - 60 x_1^2 x_2 x_3 x_4 + 
   9 x_1^3 x_2 x_3 x_4 + 6 x_1 x_2^2 x_3 x_4 + 15 x_1^2 x_3^2 x_4 + 
   6 x_2 x_3^2 x_4 - 3 x_1^2 x_2 x_4^2)\,,\\
\dot{x}_4&=&-x_1 x_4 (-11 x_1 x_2 - 15 x_1^2 x_3 + 60 x_1 x_2 x_3 - 15 x_2^2 x_3 - 
   9 x_1 x_2 x_3^2 - 3 x_1 x_4+ \\ &\ & + 3 x_1 x_2^2 x_4 - 6 x_3 x_4 + 6 x_1^2 x_3 x_4 - 
   60 x_1 x_2 x_3 x_4 + 9 x_1^2 x_2 x_3 x_4 + 6 x_2^2 x_3 x_4 + 3 x_1 x_3^2 x_4+ \\ &\ & + 
   11 x_1 x_2 x_4^2 + 15 x_3 x_4^2)\,,\\
\dot{x}_5&=&0\,.   
\end{eqnarray*}
Similarly with before, last equation confirms that infinity remains invariant under the flow of the system.
Now, by solving the system $\{\dot{x}_1=\dot{x}_2=\dot{x}_3=\dot{x}_4=0\}$, we get exactly six fixed points at the infinity of $\M^G$, given by:
\begin{eqnarray*}
&& (2,3,4,5),\ (0.599785,1.08371,0.901823,1.22291), \\
&& (1.02137,0.546007,1.05352,1.10879),\ (1.08294,1.04088,0.532615,1.10351), \\
&& (0.720713,1.02546,0.475234,1.07095),\ (1.03732,1.04718,1.03082,0.29862).
\end{eqnarray*}
As before, these fixed points  induce via the central projection $F$ the  explicit presentations $\fr{e}_i,\ i=1,\ldots,6$. The fixed point represented by $\fr{e}_1=(1,2,3,4,5)$ corresponds to the unique invariant K\"{a}hler-Einstein metric on $M=\E_8/\U(1)\times \SU(4)\times \SU(5)$ and  eigenvalues calculations show that it possesses a 4-dimensional stable manifold, contained in the infinity of $\M^G$, and a 1-dimensional unstable manifold which coincide with  a straight tending to the origin. The other three fixed points, corresponding to non-K\"ahler, non-isometric invariant Einstein metrics (see \cite{CS1}), have a 3-dimensional stable manifold, located in the infinity of $\M^G$, and a 2-dimensional unstable manifold, one direction of which is the straight line in $\M^G$ tending to the origin, with the exception of the fixed point $\fr{e}_5$, which possesses a 2-dimensional stable manifold, contained in the infinity of $\M^G$, and a 3-dimensional unstable manifold. The ancient solutions $g_i(t)=(1-2\lambda_it)\cdot\fr{e}_i$, $i=1, \ldots, 6$ are defined on the open  interval $(-\infty,1/2\lambda_i)$, where the corresponding Einstein constant $\lambda_i$ is given by
\begin{eqnarray*}
\lambda_1=11/60,  &\text{for}&  \fr{e}_1=(1,2,3,4,5),\\
\lambda_2=0.37877 , &\text{for}& \fr{e}_2=(1,0.599785, 1.08371, 0.901823, 1.22291),\\
\lambda_3=0.365507,  &\text{for}&  \fr{e}_3=(1,1.02137, 0.546007, 1.05352, 1.10879),\\
\lambda_4=0.339394, &\text{for}& \fr{e}_4=(1,1.08294, 1.04088, 0.532615, 1.10351),\\
\lambda_5=0.386982,  &\text{for}& \fr{e}_5=(1,0.720713, 1.02546, 0.475234, 1.07095),\\
\lambda_6=0.337271, &\text{for}& \fr{e}_6=(1,1.03732, 1.04718, 1.03082, 0.29862).
\end{eqnarray*}
Using these constants and by taking limits, as above, we obtain the rest claims for $r=5$.\\
\noindent  {\bf Case} $\bold{r=6.}$ To reduce the system (\ref{hrf})   to a   polynomial dynamical system, a short computation shows that we must multiply with  the positive   term $60x_{1}^2 x_{2}^2 x_{3}^2x_{4}x_{5}x_6$. This gives the following:
{\small \begin{eqnarray*}
\dot{x}_1&=&-x_1x_2x_3(x_1^3 x_2 x_3 x_4 - x_1 x_2 x_3 x_4 x_5^2 + 2 x_1^3 x_2 x_3 x_6 - 
   2 x_1 x_2 x_3 x_4^2 x_6+\\ &\ & + 4 x_1^3 x_2 x_5 x_6 - 4 x_1 x_2 x_3^2 x_5 x_6 + 
   6 x_1^3 x_4 x_5 x_6 - 6 x_1 x_2^2 x_4 x_5 x_6 + 60 x_1 x_2 x_3 x_4 x_5 x_6-\\ &\ & - 
   8 x_2^2 x_3 x_4 x_5 x_6 - 6 x_1 x_3^2 x_4 x_5 x_6 - 4 x_1 x_2 x_4^2 x_5 x_6 - 
   2 x_1 x_2 x_3 x_5^2 x_6 - x_1 x_2 x_3 x_4 x_6^2)\,,\\
\dot{x}_2&=&2x_2x_3(-x_1^2 x_2^3 x_3 x_5 + x_1^2 x_2 x_3 x_4^2 x_5 - x_1^2 x_2^3 x_4 x_6 + x_1^2 x_2 x_3^2 x_4 x_6+\\ &\ & + 3 x_1^3 x_2 x_4 x_5 x_6 - 3 x_1 x_2^3 x_4 x_5 x_6 - 
   26 x_1^2 x_2 x_3 x_4 x_5 x_6 - 2 x_2^3 x_3 x_4 x_5 x_6+\\ &\ & + 
   3 x_1 x_2 x_3^2 x_4 x_5 x_6 + 3 x_1^2 x_3 x_4^2 x_5 x_6 + x_1^2 x_2 x_4 x_5^2 x_6 +
    x_1^2 x_2 x_3 x_5 x_6^2)\,,\\
\dot{x}_3&=&3 x_1x_2x_3x_6 (x_1 x_2^2 x_3 x_4 - x_1 x_3^3 x_4 + 2 x_1^2 x_2 x_3 x_5 - 
   2 x_2 x_3^3 x_5 + 3 x_1^2 x_3 x_4 x_5-\\ &\ & - 20 x_1 x_2 x_3 x_4 x_5 + 
   3 x_2^2 x_3 x_4 x_5 - 3 x_3^3 x_4 x_5 + 2 x_2 x_3 x_4^2 x_5 + x_1 x_3 x_4 x_5^2 + 
   x_1 x_2 x_4 x_5 x_6)\,,\\
\dot{x}_4&=&2x_1x_3x_4(2 x_1 x_2^3 x_3 x_5 - 2 x_1 x_2 x_3 x_4^2 x_5 + 
   2 x_1^2 x_2^2 x_3 x_6 - 2 x_2^2 x_3 x_4^2 x_6+\\ &\ & + 4 x_1^2 x_2^2 x_5 x_6 - 
   24 x_1 x_2^2 x_3 x_5 x_6 + 4 x_2^2 x_3^2 x_5 x_6 - 4 x_2^2 x_4^2 x_5 x_6 - 
   3 x_1 x_3 x_4^2 x_5 x_6+\\ &\ & + 2 x_2^2 x_3 x_5^2 x_6 + 2 x_1 x_2 x_3 x_5 x_6^2)\,,\\
\dot{x}_5&=&5x_1x_2x_3x_5 (x_1^2 x_2 x_3 x_4 - x_2 x_3 x_4 x_5^2 + 2 x_1^2 x_2 x_3 x_6 + 
   2 x_1 x_2^2 x_4 x_6-\\ &\ & - 12 x_1 x_2 x_3 x_4 x_6 + 2 x_1 x_3^2 x_4 x_6 + 
   2 x_2 x_3 x_4^2 x_6 - 2 x_2 x_3 x_5^2 x_6 - 2 x_1 x_4 x_5^2 x_6 + 
   x_2 x_3 x_4 x_6^2)\,,\\
   \dot{x}_6&=&6x_1x_2x_6(x_1^2 x_2 x_3^2 x_4 + 2 x_1 x_2^2 x_3^2 x_5 - 
   8 x_1 x_2 x_3^2 x_4 x_5 + 2 x_1 x_3^2 x_4^2 x_5+\\ &\ & + x_2 x_3^2 x_4 x_5^2 - 
   x_2 x_3^2 x_4 x_6^2 - 2 x_1 x_3^2 x_5 x_6^2 - x_1 x_2 x_4 x_5 x_6^2)\,.  
\end{eqnarray*}}
Note that the right-hand side consists of homogeneous polynomials of degree $9$.
Hence again we can apply  the  Poincar\'{e} compactification procedure  to study the behaviour of this system at infinity of $\M^G$. Using the expressions given above, the system at infinity, written in the $U$ chart and setting $x_6=0$, reads as follows:
{\small \begin{eqnarray*}
\dot{x}_1&=&-x_1 x_2 (-x_1^2 x_2 x_3 + 2 x_1^3 x_2 x_4 - 2 x_1 x_2 x_3^2 x_4 + 
   x_1^2 x_2 x_3 x_4^2 - 2 x_1^2 x_2 x_5 + 2 x_1^3 x_3 x_5 -\\ &\ &- 2 x_1 x_2^2 x_3 x_5 + 
   2 x_1^2 x_2 x_3^2 x_5 - 4 x_1^2 x_4 x_5 + 4 x_1^2 x_2^2 x_4 x_5 - 
   12 x_1 x_3 x_4 x_5+\\ &\ & + 12 x_1^3 x_3 x_4 x_5 + 52 x_1 x_2 x_3 x_4 x_5 - 
   60 x_1^2 x_2 x_3 x_4 x_5 + 12 x_1^3 x_2 x_3 x_4 x_5 + 4 x_1^2 x_3^2 x_4 x_5-\\ &\ & - 
   6 x_2 x_3^2 x_4 x_5 + 2 x_1^2 x_2 x_4^2 x_5 - 2 x_1 x_3 x_4^2 x_5 + 
   x_1^2 x_2 x_3 x_5^2 - 2 x_1 x_2 x_4 x_5^2)\,,\\
\dot{x}_2&=&-x_1 x_2 (-x_1 x_2^2 x_3 + x_1 x_2^2 x_3 x_4^2 - 2 x_1 x_2^2 x_5 - 
   3 x_1^2 x_2 x_3 x_5 + 3 x_2^3 x_3 x_5+\\ &\ & + 2 x_1 x_2^2 x_3^2 x_5 - 
   10 x_1 x_2 x_4 x_5 + 10 x_1 x_2^3 x_4 x_5 - 15 x_2 x_3 x_4 x_5 + 
   60 x_1 x_2 x_3 x_4 x_5-\\ &\ & - 3 x_1^2 x_2 x_3 x_4 x_5 - 60 x_1 x_2^2 x_3 x_4 x_5 + 
   8 x_1^2 x_2^2 x_3 x_4 x_5 + 15 x_2^3 x_3 x_4 x_5 -\\ &\ &- 2 x_1 x_2 x_3^2 x_4 x_5 + 
   2 x_1 x_2^2 x_4^2 x_5 - 3 x_2 x_3 x_4^2 x_5 + x_1 x_2^2 x_3 x_5^2 - 
   3 x_1 x_3 x_4 x_5^2)\,,\\
\dot{x}_3&=&-x_2 x_3 (-x_1^2 x_2 x_3 - 4 x_1^3 x_2 x_4 + 4 x_1 x_2 x_3^2 x_4 + 
   x_1^2 x_2 x_3 x_4^2 - 6 x_1^2 x_2 x_5+\\ &\ & + 6 x_1^2 x_2 x_3^2 x_5 - 
   12 x_1^2 x_4 x_5 + 48 x_1^2 x_2 x_4 x_5 - 4 x_1^2 x_2^2 x_4 x_5 - 
   6 x_1 x_3 x_4 x_5+\\ &\ & + 6 x_1^3 x_3 x_4 x_5 - 60 x_1^2 x_2 x_3 x_4 x_5 + 
   8 x_1^3 x_2 x_3 x_4 x_5 + 6 x_1 x_2^2 x_3 x_4 x_5 +\\ &\ &+ 12 x_1^2 x_3^2 x_4 x_5 + 
   6 x_2 x_3^2 x_4 x_5 - 2 x_1^2 x_2 x_4^2 x_5 + x_1^2 x_2 x_3 x_5^2 - 
   4 x_1 x_2 x_4 x_5^2)\,,\\
\dot{x}_4&=&-2 x_1 x_2 x_4 (-3 x_1 x_2 x_3 + 3 x_1 x_2 x_3 x_4^2 - 6 x_1 x_2 x_5 - 
   5 x_1^2 x_3 x_5 + 30 x_1 x_2 x_3 x_5-\\ &\ & - 5 x_2^2 x_3 x_5 - 4 x_1 x_2 x_3^2 x_5 - 
   2 x_1 x_4 x_5 + 2 x_1 x_2^2 x_4 x_5 - 3 x_3 x_4 x_5 + 3 x_1^2 x_3 x_4 x_5-\\ &\ & - 
   30 x_1 x_2 x_3 x_4 x_5 + 4 x_1^2 x_2 x_3 x_4 x_5 + 3 x_2^2 x_3 x_4 x_5 + 
   2 x_1 x_3^2 x_4 x_5 + 6 x_1 x_2 x_4^2 x_5+\\ &\ & + 5 x_3 x_4^2 x_5 - 2 x_1 x_2 x_3 x_5^2)\,,\\
\dot{x}_5&=&-x_1 x_5 (-7 x_1 x_2^2 x_3 - 12 x_1^2 x_2^2 x_4 + 48 x_1 x_2^2 x_3 x_4 - 
   12 x_2^2 x_3^2 x_4 - 5 x_1 x_2^2 x_3 x_4^2-\\ &\ & - 2 x_1 x_2^2 x_5 + 
   2 x_1 x_2^2 x_3^2 x_5 - 4 x_1 x_2 x_4 x_5 + 4 x_1 x_2^3 x_4 x_5-\\ &\ & - 
   6 x_2 x_3 x_4 x_5 + 6 x_1^2 x_2 x_3 x_4 x_5 - 60 x_1 x_2^2 x_3 x_4 x_5 + 
   8 x_1^2 x_2^2 x_3 x_4 x_5 +\\ &\ &+ 6 x_2^3 x_3 x_4 x_5 + 4 x_1 x_2 x_3^2 x_4 x_5 + 
   2 x_1 x_2^2 x_4^2 x_5 + 7 x_1 x_2^2 x_3 x_5^2+\\ &\ & + 12 x_2^2 x_4 x_5^2 + 
   6 x_1 x_3 x_4 x_5^2)\,,\\
   \dot{x}_6&=&0\,.  
\end{eqnarray*}}
By the  last equation one deduces that the infinity of $\M^G$ remains invariant under the flow of the system.
Solving now the system $\{\dot{x}_1=\dot{x}_2=\dot{x}_3=\dot{x}_4=\dot{x}_5=0\}$, we get exactly  five fixed points given by:
$$(2,3,4,5,6),\ (0.823084,1.1467,1.17377,1.42664,1.46519),$$
$$(0.986536,0.636844,1.06853,1.13323,0.921127),$$
$$(0.90422,0.778283,0.927483, 1.03408,0.359949),$$
$$(0.954875,0.965321,1.00534,0.290091,1.01965).$$
These solutions, projected through the central projection $F$, induce the fixed points $\fr{e}_i$ at infinity of $\M^G$,  namely
\begin{eqnarray*}
&&\fr{e}_1=(1,2,3,4,5,6),\\
&& \fr{e}_2=(1,0.823084,1.1467,1.17377,1.42664,1.46519),\\
&& \fr{e}_3=(1,0.986536,0.636844,1.06853,1.13323,0.921127),\\
&& \fr{e}_4=(1,0.90422,0.778283, 0.927483,1.03408,0.359949),\\
&& \fr{e}_5=(1,0.954875,0.965321,1.00534,0.290091,1.01965). 
\end{eqnarray*}
These points are in bijective correspondence with  non-isometric invariant Einstein metrics on $M=G/K=\E_8/\U(1)\times \SU(2)\times \SU(3)\times \SU(5)$ (see \cite{CS1}), and determine the non-collapsed ancient solutions $g_i(t)=(1-2\lambda_it)\cdot\fr{e}_i$, where  the corresponding Einstein constant $\lambda_i$ has the form
\[
\lambda_1=3/20,\quad \lambda_2=0.313933,\quad \lambda_3=0.348603,\quad \lambda_4=0.367518,\quad  \lambda_5=0.349296.
\]
All the other conclusions are obtained  in an analogous way as before. This completes the proof.
\end{proof}

\begin{remark}
 For $r\geq  3$ a statement of $\omega$-limits of general solutions of (\ref{hrf}), as in the assertion (4) of Theorem \ref{mainthem},  is not presented, since for these cases  a Lyapunov function is hard to be computed. %Hence,   it remains an {open question} the behaviour of $\omega$-limits for general homogeneous solutions when $3\leq r\leq 6$.
\end{remark}
Let us  now take a view of the limit behaviour of the scalar curvature for the solutions $g_{i}(t)$  given  in Theorem \ref{mainthem}, and 
verify the statements of Proposition \ref{ak}. We do this by treating examples with $r=2, 3$.
\begin{example}\label{s2c}
Let $M=G/K$ be a flag manifold with $r=2$. Then, an application of (\ref{sca2c})  shows that both $\Scal(g_{i}(t))$ are  {\it positive hyperbolas}, given by
\begin{eqnarray*}
\Scal(g_{1}(t))&=&\frac{(d_1 + d_2) (d_1 + 2 d_2)}{2(d_1+4d_2)-2t(d_1+2d_2)}\,,\\ 
\Scal(g_2(t))&=&\frac{(d_1 + d_2) (d_1^2 + 6 d_1 d_2 + 4 d_2^2)}{2\big((d_1^2+6d_1d_2+8d_2^2)-t(d_1^2+6d_1d_2+4d_2^2)\big)}\,,
\end{eqnarray*}
respectively.  Therefore, they both are increasing in   $(-\infty, \frac{1}{2\lambda_i})$, which is the open interval which are defined, i.e.  $\Scal'(g_{1}(t))$ and $\Scal'(g_{2}(t))$ are strictly positive. 
The limit of $\Scal(g_{i}(t))$ as $t\to \frac{1}{2\lambda_i}$  must be considered only  from below, and it is direct to check that 
\[
 \lim_{t\to\frac{1}{2\lambda_i}}\Scal(g_{i}(t))=+\infty\,,\quad\lim_{t\to-\infty}\Scal(g_{i}(t))=0\,,\quad\forall \ i=1, 2.
\]
 Hence,   $\Scal(g_{i}(t))>0$, for any  $t\in(-\infty, \frac{1}{2\lambda_i})$ and for any $i=1, 2$, as it should be for ancient solutions in combination with the non-existence of  Ricci flat metrics  (\cite{Kot10, Laf15}).   Let us present the graphs of $\Scal(g_1(t))$ and $\Scal(g_{2}(t))$   for the flag spaces $\G_2/\U(2)$ and  $\F_4/\Sp(3)\times\U(1)$, both with $r=2$. We list all  the related details  below, together  with the graphs of $\Scal(g_{i}(t))$  for the corresponding  intervals of definition $(-\infty, \frac{1}{2\lambda_i})$ (see Figure \ref{fig:1} and \ref{fig:2}, for $i=1, 2$, respectively).
\begin{table}[ht]
\centering
\renewcommand\arraystretch{2.2}
 \begin{tabular}{c | c | c | c | c | c | c  }
 $M=G/K$ & $d_1$ & $d_2$ & $1/2\lambda_1$ & $1/2\lambda_2$ & $\Scal(g_1(t))$ & $\Scal(g_2(t))$ \\
 \thickline
 $\G_2/\U(2)$ & 8 & 2 & 4/3 & 12/11 & $\displaystyle\frac{120}{32 - 24 t}$ & $\displaystyle\frac{1760}{384-362t}$ \\
 $\F_4/\Sp(3)\times\U(1)$  & 28 & 2 & 9/8 & 72/71 &  $\displaystyle\frac{960}{72 - 64 t}$ & $ \displaystyle\frac{34080}{2304-2272t}$\\
 \hline
\end{tabular}
\end{table}
 \begin{figure}[h!]
 \centering
   {\small{ \includegraphics[scale=0.5]{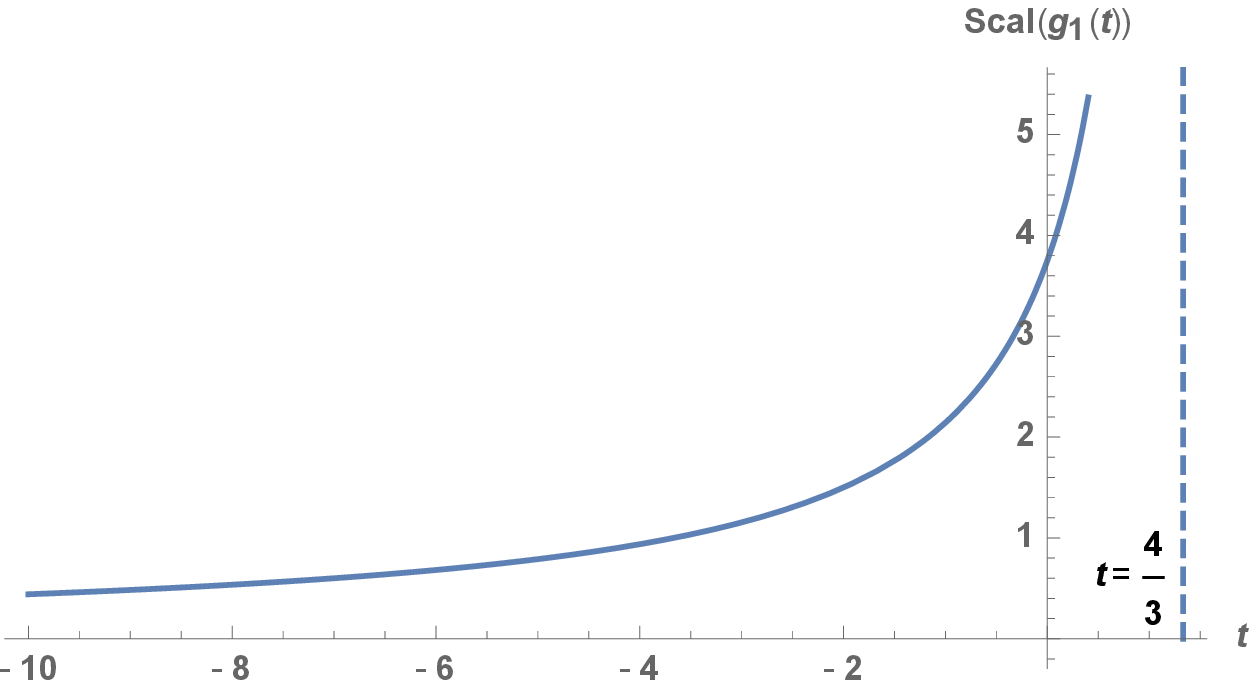}}} \ \ \  \ \ \  \ \ \  {\small{ \includegraphics[scale=0.5]{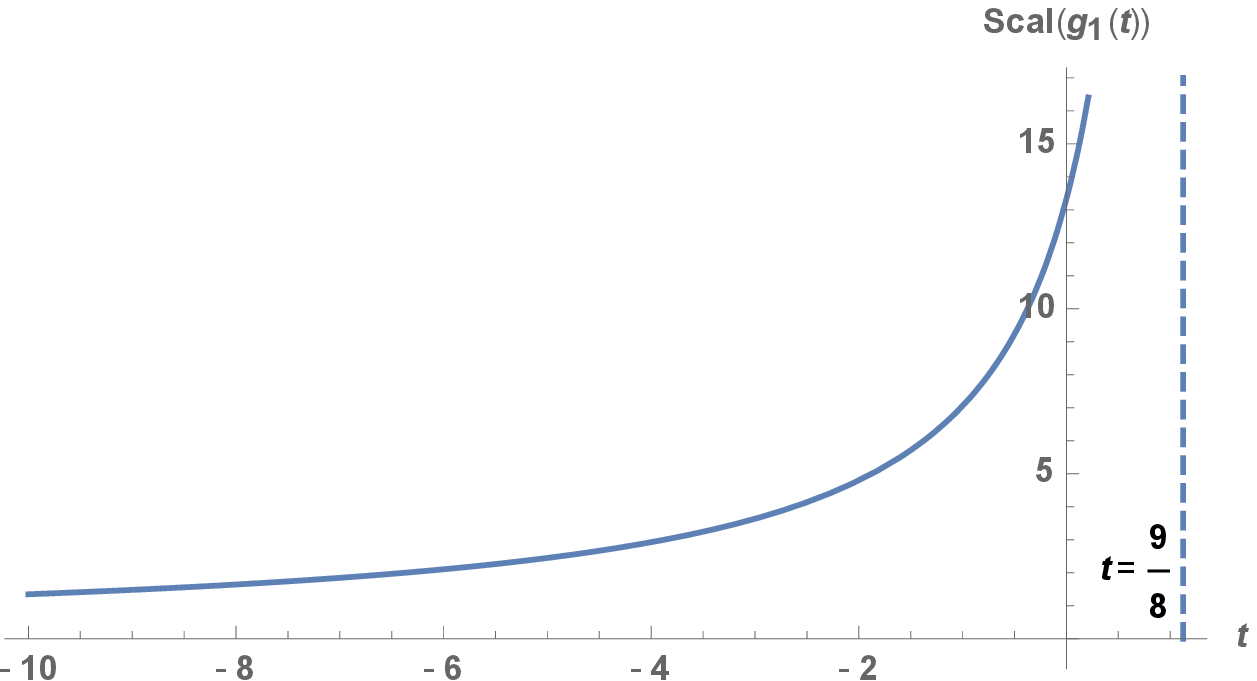}}}
     \caption{The graph of $\Scal(g_{1}(t))$ for  $\G_2/\U(2)$ (left) and $\F_4/\Sp(3)\times\U(1)$ (right).}
    \label{fig:1}
    \end{figure}
   \begin{figure}[h!]
 \centering
   {\small{ \includegraphics[scale=0.5]{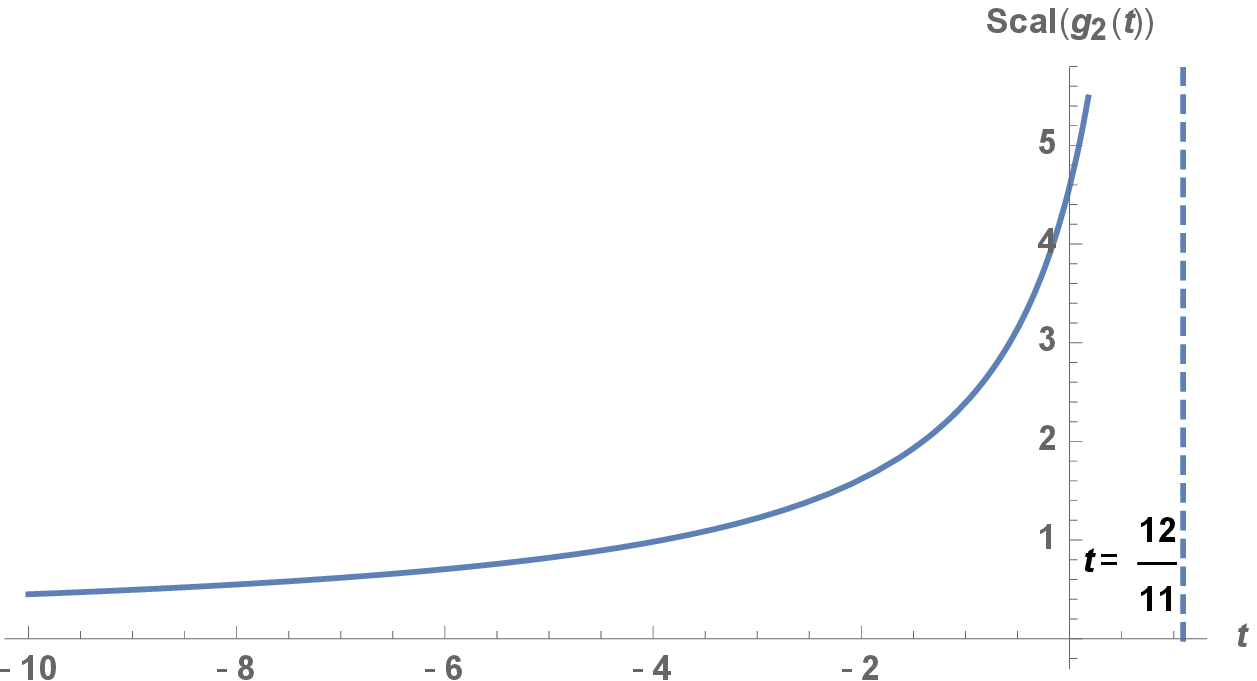}}} \ \ \  \ \ \  \ \ \  {\small{ \includegraphics[scale=0.5]{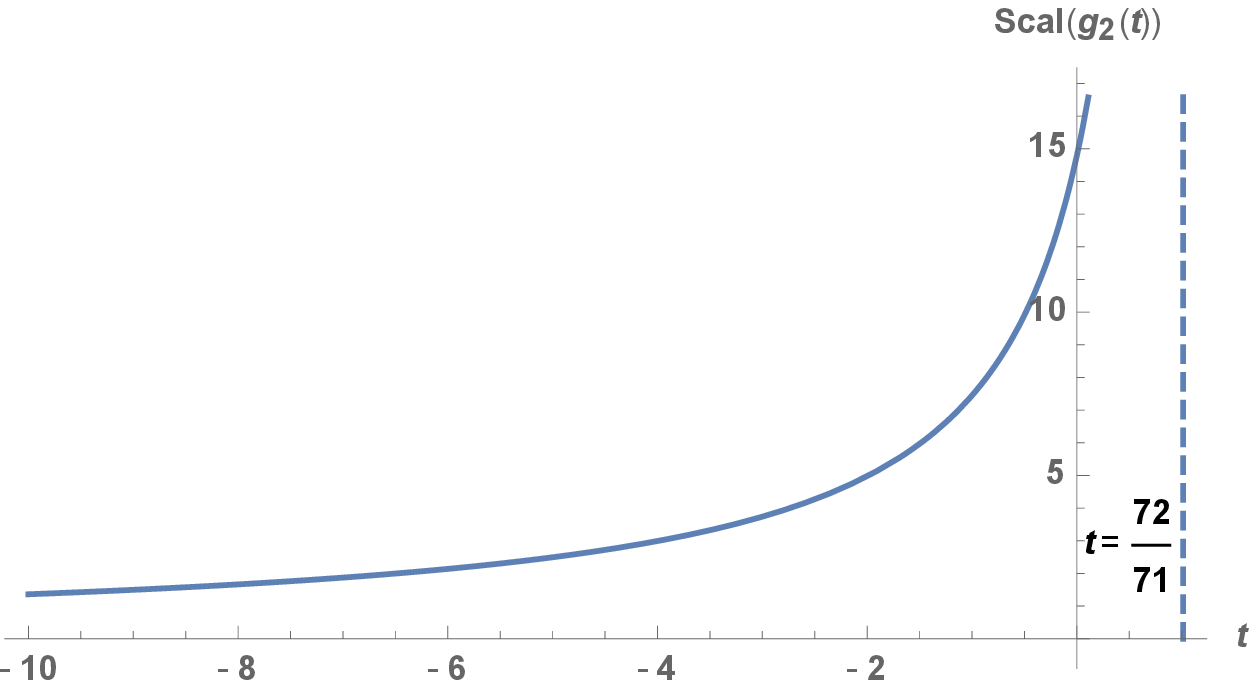}}}
     \caption{The graph of $\Scal(g_{2}(t))$ for  $\G_2/\U(2)$ (left) and $\F_4/\Sp(3)\times\U(1)$ (right).}
    \label{fig:2}
    \end{figure}

\noindent    The graphs now of the ricci components of the solutions $g_{j}(t)$,which  we denote by 
    \[
    \ric_{ij}(t):=\ric_{i}(g_{j}(t))\,,\quad   1\leq i\leq r\,, \quad  1\leq j\leq N, 
    \]
 where again $N$ represents the number of fixed points of HRF at infinity of $\M^G$,     are very similar. For example, for any  flag manifold  $M=G/K$ with $r=2$ and for the first ancient solution $g_{1}(t)$ passing through the invariant  K\"ahler-Einstein metric $g_{1}(0)=g_{KE}$, we compute
    \[
 \ric_{11}(t)=\ric_1(g_1(t))=\ric_{21}(t)=\ric_2(g_1(t))=\frac{d_1 + 2 d_2}{2 (d_1 + 4d_2) - 2 (d_1 + 2 d_2) t}\,.
    \]
    For the solution $g_{2}(t)$ passing from the non-K\"ahler invariant Einstein metric $g_{2}(0)=g_{E}$ we obtain
      \[
 \ric_{12}(t)=\ric_1(g_2(t))=\ric_{22}(t)=\ric_2(g_2(t))=\frac{d_1^2 + 6 d_1 d_2 + 4 d_2^2}{
 2(d_1^{2}+6d_1d_2+8d_{2}^{2}) -2(4d_2^2 + d_1^2 + 6 d_1 d_2)t}\,.
    \]
    Note that the equalities $\ric_{11}(t)=\ric_{21}(t)$ and  $\ric_{12}(t)=\ric_{22}(t)$ occur since $g_{i}(t)$ $(i=1, 2)$ are both 1-parameter families of invariant Einstein metrics on $M=G/K$ (as we mentioned in Section \ref{ancient}).
For instance, for $\G_2/\U_2$ with $r=2$, the above formulas reduce to
    \[
    \ric_{11}(t)=\ric_{21}(t)=\frac{12}{32 - 24 t}\,,\quad  \ric_{12}(t)=\ric_{22}(t)=\frac{88}{192-176t} \,,
    \]
    with $\ric_{11}(0)=\ric_{21}(0)=\lambda_1=3/8$ and $ \ric_{12}(0)=\ric_{22}(0)=\lambda_2=11/24$, respectively. The corresponding graphs are given in Figure \ref{fig:1/2}.
     \begin{figure}[h!]
 \centering
   {\small{ \includegraphics[scale=0.4]{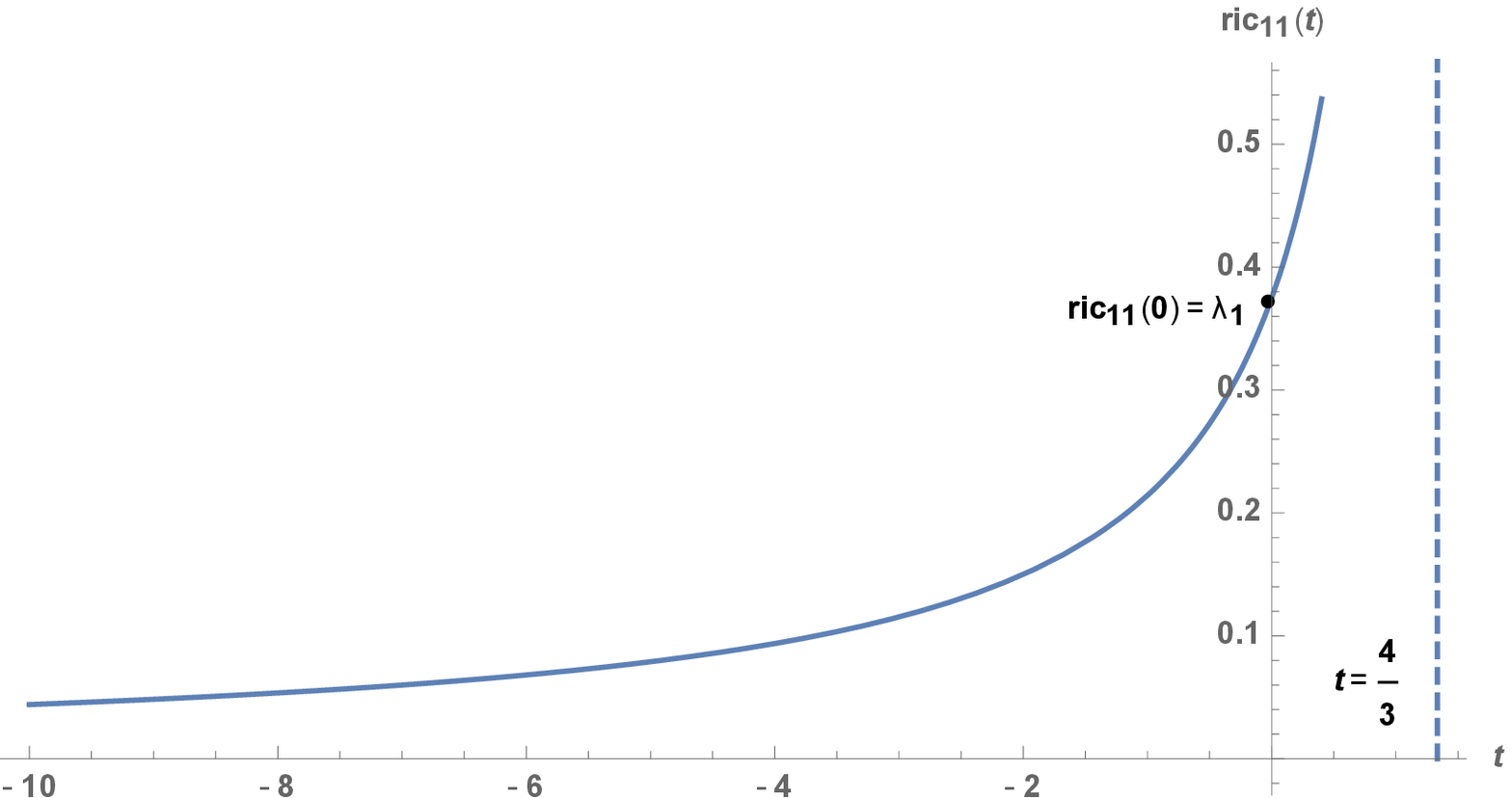}}} \    {\small{ \includegraphics[scale=0.4]{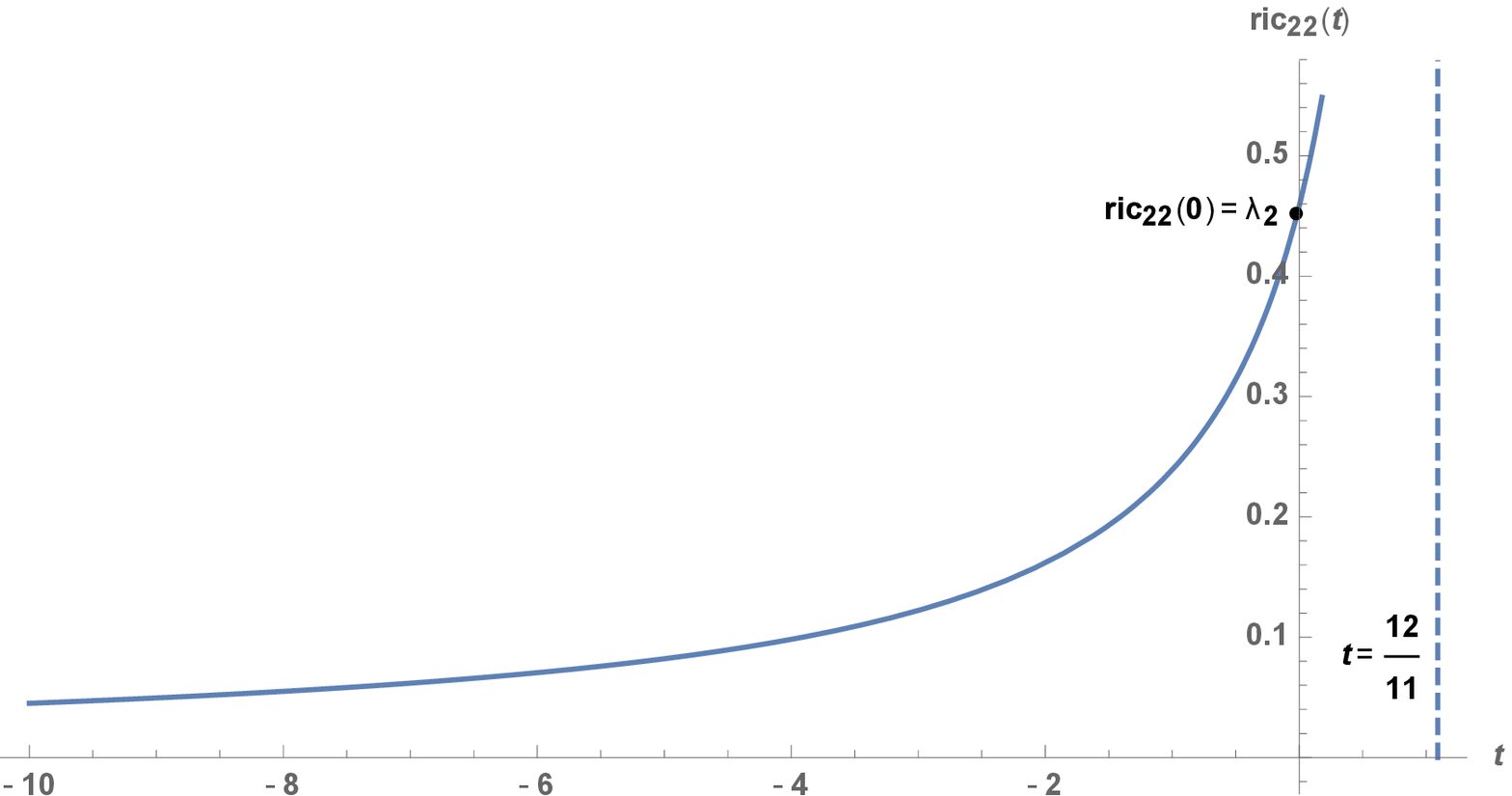}}}
     \caption{The graphs of $\ric_{11}(t)=\ric_{21}(t)$  and $\ric_{12}(t)=\ric_{22}(t)$ for  $\G_2/\U(2)$ with $\fr{m}=\fr{m}_1\oplus\fr{m}_2$.}
    \label{fig:1/2}
    \end{figure}
    \end{example}

\begin{example}\label{3compexa}
Let $M=G/K$ be a flag manifold with $b_{2}(M)=1$ and $r=3$.  Let us describe the  asymptotic properties of the scalar curvature  $\Scal(g_{1}(t))$, related  to the solution  $g_{1}(t)=(1-2\lambda_{1}t)\cdot\fr{e}_{1}$ only, where $\fr{e}_1=(1, 2, 3)$ is the fixed point corresponding to the invariant K\"ahler-Einstein metric $g_{1}(0)\equiv g_{KE}$.  By (\ref{sca3c}) we see that $\Scal(g_{1}(t))$ is the positive hyperbola given by
\[
\Scal(g_{1}(t))=\frac{(d_1 + d_2 + d_3)(d_1 + 2 d_2 + 3 d_3)}{2(d_1+4d_2+9d_3)+2t(d_1+4d_2+6d_3)}\,.
% 6 d3 (-3 + t) + 4 d2 (-2 + t) + 2 d1 (-1 + t)}
\]
Thus,  $\Scal(g_{1}(t))$ increases  on the open interval $(-\infty, \frac{1}{2\lambda_1})$, where $g_{1}(t)$ is dedined. Note that the value $\Scal(g(0))=\frac{(d_1 + d_2 + d_3) (d_1 + 2 d_2 + 3 d_3)}{2(d_1 + 4 d_2 +9 d_3)}$ equals to the scalar curvature of the K\"ahler-Einstein metric $g_1(0)$. The limit of $\Scal(g_{1}(t))$ as $t\to \frac{1}{2\lambda_1}$  must be considered only  from below, and it follows that
\[
\lim_{t\to\frac{1}{2\lambda_1}}\Scal(g_{1}(t))=+\infty\,,\quad \lim_{t\to-\infty}\Scal(g_{1}(t))=0\,.
\]
For example for $\G_2/\U_2$ and $r=3$, we compute $\lambda_1=5/24$, so
\[
\Scal(g_{1}(t))=\frac{25}{12-5t}\,,\quad t\in(-\infty, \frac{12}{5})\,, \quad \Scal(g_{1}(0))=\frac{25}{12}
\]
and the graph of $\Scal(g_{1}(t))$ is given by
%Thuss, $\Scal(g_{i}(t))>0$, for any  $t\in(-\infty, \frac{1}{2\lambda_i})$ and for any $i=1, 2$.  
 \begin{figure}[h!]
 \centering
   {\small{ \includegraphics[scale=0.5]{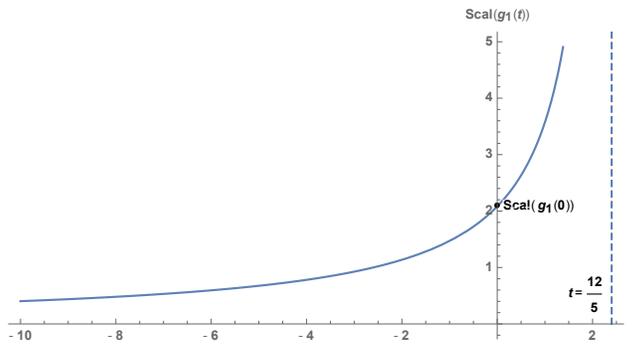}}}
     \caption{The graph of $\Scal(g_{1}(t))$ for  $\G_2/\U(2)$ with $\fr{m}=\fr{m}_1\oplus\fr{m}_2\oplus\fr{m}_3$.}
    \label{fig:3}
    \end{figure}
\end{example}

\end{document}